\documentclass{amsart}

\usepackage{amsmath}
\usepackage{amssymb}
\usepackage{mathrsfs}
\usepackage{tikz}
\usetikzlibrary{arrows.meta,positioning}

\newtheorem{theorem}{Theorem}[section]
\newtheorem{lemma}[theorem]{Lemma}
\newtheorem{proposition}[theorem]{Proposition}
\newtheorem{corollary}[theorem]{Corollary}

\theoremstyle{definition}
\newtheorem{definition}[theorem]{Definition}

\theoremstyle{remark}
\newtheorem{remark}[theorem]{Remark}

\numberwithin{equation}{section}

\newcommand{\R}{\mathbb{R}}
\newcommand{\N}{\mathbb{N}}
\newcommand{\GRP}{\mathrm{GRP}}

\begin{document}

\title[Banach-valued graph limits]{Banach-valued graph limits:\\ Graphon representability and Banach-space structure}

\author[M. Otsuka]{Motoki Otsuka}
\address{University of Tsukuba, 1-1-1, Tennodai, 305-8577, Ibaraki, Japan}
\email{s2220082@u.tsukuba.ac.jp}


\date{July 29, 2026}


\begin{abstract}
We study a graph-limit problem for Banach-decorated graphs. Given a sequence of \(X\)-decorated graphs whose homomorphism densities converge against all \(X^*\)-decorated test graphs, we ask whether the limiting densities are represented by an \(X\)-valued graphon. The results connect this graph-limit problem with Banach-space structure.

If \(X^*\) is separable, then the graphon representation property for graph sequences uniformly bounded in \(L^p\) for every finite \(p\) holds if and only if \(X\) is reflexive. For Banach lattices, it is equivalent to the Radon--Nikod\'ym property. For dual Banach spaces, it is equivalent to the conjunction of the Radon--Nikod\'ym property and weak sequential completeness. In the bounded setting, the same characterization extends to arbitrary Banach spaces: for every Banach space \(X\), the representation property for uniformly \(L^\infty\)-bounded graph sequences holds if and only if \(X\) has the Radon--Nikod\'ym property and is weakly sequentially complete.
\end{abstract}

\maketitle

\section{Introduction}

In the theory of dense graph limits, a sequence of finite simple graphs whose
numbers of vertices tend to infinity is studied through the homomorphism
densities of fixed finite graphs.  For a finite simple graph \(F\), the
homomorphism density \(t(F,G_n)\) is the probability that a uniformly chosen
map from the vertex set \(V(F)\) of \(F\) to the vertex set \(V(G_n)\) of
\(G_n\) preserves edges.  A graph sequence is said to converge when these
densities converge for every \(F\).  Lov\'asz and Szegedy proved that the
resulting limiting densities can all be represented simultaneously by a
scalar graphon, namely a symmetric measurable function
\(w:[0,1]^2\to[0,1]\), and conversely that every scalar graphon arises as the
limit of a convergent sequence of finite simple graphs
\cite{lovaszSzegedy2006dense}.

The representation theorem is the starting point of a broader theory linking
finite graphs, graphons, and random sampling.  Every finite graph can be
represented as a graphon by partitioning \([0,1]\) into equal-length intervals
indexed by its vertices.  The cut distance provides a common way to compare
finite graphs and graphons.  After graphons at cut distance zero are
identified, the graphon space with the cut distance is compact.  In this space,
convergence in cut distance, convergence of all homomorphism densities, and
convergence in distribution of sampled subgraphs of every fixed size are
equivalent
\cite{borgsChayesLovaszSosVesztergombi2008denseI}.  Under any of these
equivalent notions of convergence, the limiting graphon is unique up to weak
isomorphism, the measure-preserving counterpart of vertex relabeling for
graphons
\cite{borgsChayesLovasz2010uniqueness}.  This framework has found applications,
for example, in property testing, exchangeable random graphs, statistical
physics, and large-deviation theory
\cite{lovaszSzegedy2010testing,diaconisJanson2007exchangeable,
borgsChayesLovaszSosVesztergombi2012denseII,
chatterjeeVaradhan2011largeDeviations}.

Weighted graphs and multigraphs provide a basic motivation to extend graph limit
theory beyond simple graphs.  In a weighted graph, each edge carries a real
weight.  In a multigraph, each pair of vertices has an edge multiplicity,
equal to the number of parallel edges between them.  Both models can be
treated as decorated graphs, in which each edge is labeled by a point of a
fixed space.  Lov\'asz and Szegedy developed this theory for a compact
second-countable decoration space.  They defined homomorphism densities using
test graphs whose edges are decorated by continuous functions on that space.
They proved that every convergent sequence of such graphs has a limit
represented by a graphon taking values in the probability measures on the
decoration space \cite{lovaszSzegedy2010compactDecorated}.  This compact
framework covers weighted graphs with bounded edge weights and multigraphs
with bounded edge multiplicities.  Edge-colored graphs are another special
case.

Kunszenti-Kov\'acs, Lov\'asz, and Szegedy extended this classical bounded
theory to a Banach-space framework allowing unbounded decorations.  In
particular, it can handle unbounded edge weights and
unbounded edge multiplicities.  Let \(B\) be a separable Banach space and
\(B^*\) its dual.  They consider \(B^*\)-valued graphons and define their
homomorphism densities using test graphs whose edges are decorated by elements
of \(B\).  Their main theorem starts with a sequence of such graphons that is
uniformly bounded in \(L^p\) for every finite \(p\) and whose homomorphism
densities converge.  It then gives a \(B^*\)-valued graphon that represents all
limiting densities \cite[Theorem~3.7]{kls2022decorated}.  More precisely, it is
enough to assume convergence for a countable generating family of test graphs.
Taking graphon values in \(B^*\) provides the weak-* measurability and weak-*
compactness needed for the representation argument.  On this basis, the
representation theorem is established in the predual-test/dual-target
framework: the predual \(B\) decorates the test graphs, and the dual \(B^*\) is
the graphon range.

In this paper, we study another natural framework: the target graphs and
graphons take values in a Banach space \(X\), while the test graphs are
decorated by elements of its dual \(X^*\).  Kunszenti-Kov\'acs, Lov\'asz, and
Szegedy also discuss this framework.  They point out that closed balls in a
general Banach space need not
be compact in the weak topology, and that weak measurability alone does not
imply Pettis integrability, even for bounded functions
\cite[Section~2.1]{kls2022decorated}.  These difficulties lead them not to
adopt this framework as their main setting.  Their theorem nevertheless has a
consequence for this framework when \(X^*\) is separable.  Under the canonical
embedding \(X\hookrightarrow X^{**}\), both graphs whose edge decorations lie
in \(X\) and graphons taking values in \(X\) can be viewed as
\(X^{**}\)-valued objects.  The theorem gives an \(X^{**}\)-valued graphon
representing the limiting densities, but not necessarily an \(X\)-valued one
\cite[Theorem~3.7]{kls2022decorated}.  This naturally raises the question
studied here: for which Banach spaces \(X\) does every such density limit admit
an \(X\)-valued graphon representation?

To formalize this question, we introduce the graphon representation property
(\(\GRP\)) for a Banach space \(X\) (Definition~\ref{def:grp}).  Consider a
sequence of target graphs whose edge decorations lie in \(X\).  We assume that
the sequence is uniformly bounded in \(L^p\) for every finite \(p\), and that
its homomorphism densities converge against every test graph whose edge
decorations lie in \(X^*\).  The \(\GRP\) requires that, for every such
sequence, all limiting densities be represented by an \(X\)-valued Bochner
measurable graphon.  The \(\GRP\) is thus a property of the Banach space itself:
it asks whether the limiting densities can always be represented in the
original space rather than only in the bidual.

Our results give characterizations of this representation property
in several major classes in the unbounded setting, and for every Banach space
in the classical uniformly bounded setting.  We prove that every Banach space
with the \(\GRP\) is weakly sequentially complete and has the
Radon--Nikod\'ym property.  Together with corresponding sufficient conditions,
these necessary conditions give characterizations for spaces with
separable dual, Banach lattices, and dual Banach spaces.  When the \(\GRP\) is
weakened to the classical uniformly bounded setting, the same two conditions
give a characterization for every Banach space.  We now describe
these results in more detail: first the general necessary conditions, then the
three structural classes, and finally the bounded setting.

Our first main result gives two general necessary conditions for the \(\GRP\).
More precisely, every Banach space with the \(\GRP\) is
weakly sequentially complete and has the Radon--Nikod\'ym property
(Proposition~\ref{prop:grp-implies-wsc} and
Corollary~\ref{cor:grp-implies-rnp}).  Conceptually, the limiting behavior of
decorated graph sequences can detect two aspects of Banach-space geometry:
whether weakly Cauchy sequences have weak limits and whether absolutely
continuous vector measures admit Radon--Nikod\'ym derivatives.  We also
establish a related structural property: the \(\GRP\) is separably
determined.  Namely, \(X\) has the
\(\GRP\) if and only if every closed separable subspace of \(X\) has it
(Theorem~\ref{thm:separable-descent}).

In three major structural classes, the necessary conditions above lead to
characterizations of the \(\GRP\).  If \(X^*\) is separable, then
\(X\) has the \(\GRP\) if and only if \(X\) is reflexive
(Theorem~\ref{thm:separable-dual-characterization}).  If \(X\) is a Banach
lattice, then \(X\) has the \(\GRP\) if and only if \(X\) has the
Radon--Nikod\'ym property
(Theorem~\ref{thm:banach-lattice-characterization}).  If \(X\) is a dual
Banach space, then \(X\) has the \(\GRP\) if and only if it is weakly
sequentially complete and has the Radon--Nikod\'ym property
(Theorem~\ref{thm:dual-banach-characterization}).  Thus, for example,
\(\ell_1\) has the \(\GRP\), whereas \(c_0\) and \(L^1([0,1])\) do not.

For the classical uniformly bounded setting, we introduce the bounded graphon
representation property (\(\mathrm{BGRP}\)), which admits a complete
characterization for every Banach space.  We say that \(X\) has the
\(\mathrm{BGRP}\) if every graph sequence that is uniformly bounded in
\(L^\infty\) and whose homomorphism densities converge has a bounded
\(X\)-valued Bochner measurable graphon representing all limiting densities
(Definition~\ref{def:bgrp}).  For every Banach space \(X\), the
\(\mathrm{BGRP}\) holds if and only if \(X\) is weakly sequentially complete
and has the Radon--Nikod\'ym property
(Theorem~\ref{thm:bounded-characterization}).  This characterization
pinpoints why the representation property can fail: if the limiting densities
of such a sequence can be represented by an \(X^{**}\)-valued graphon but not
by any \(X\)-valued (Bochner measurable) graphon, then the obstruction
must come from a failure of weak sequential completeness or of the
Radon--Nikod\'ym property in \(X\) (either failure can occur).

The remainder of the paper is organized as follows.
Section~\ref{sec:graphon-frameworks} introduces the graphon frameworks and
notation used throughout. Section~\ref{sec:necessary-conditions} establishes
the necessary conditions, Section~\ref{sec:positive-characterizations} treats
the three structural classes, and Section~\ref{sec:bounded-representation}
addresses the bounded setting.

\section{Graphon frameworks}
\label{sec:graphon-frameworks}

In this section, we recall the part of the framework of Kunszenti-Kov\'acs,
Lov\'asz, and Szegedy \cite{kls2022decorated} needed here and introduce the framework used in this
paper. We call graphons in the former framework \emph{weak-* graphons} and those in the latter
\emph{Bochner graphons}, and denote them by \(U\) and \(W\), respectively.

\subsection{Notation}

Unless otherwise specified, \([0,1]\) and, more generally, \([0,1]^S\) for a
finite set \(S\) are equipped with their Lebesgue \(\sigma\)-algebras and
Lebesgue measures.
In this paper, all Banach spaces are assumed to be real. For a measure space
\((\Omega,\Sigma,\mu)\), a Banach space \(E\), and \(1\le p\le\infty\),
\(L^p(\Omega,\Sigma,\mu;E)\) denotes the Bochner \(L^p\)-space of
\(E\)-valued functions. When \(\Sigma\) and \(\mu\) are understood, we write
\(L^p(\Omega;E)\). When \(E=\mathbb R\), we omit \(E\) and write
\(L^p(\Omega,\Sigma,\mu)\), or \(L^p(\Omega)\) when \(\Sigma\) and \(\mu\)
are understood. All equalities between measurable functions are understood up
to null sets.

For a Banach space \(E\), we write \(E^*\) for its dual and
\(\langle \phi,x\rangle:=\phi(x)\) for \(\phi\in E^*\) and \(x\in E\).
We identify \(E\) with its canonical image in \(E^{**}\). Under this
identification, for \(x\in E\) and \(\phi\in E^*\), we also write
\(\langle x,\phi\rangle:=\phi(x)\).
For Banach spaces \(E\) and \(F\), we write \(E\cong F\) if they are
isometrically isomorphic.

Unless otherwise specified, the vertex set of a \(k\)-vertex simple graph is
\(\{1,\ldots,k\}\). In a few places, we instead use
\(\{0,\ldots,k-1\}\).

\subsection{Weak-* graphons}

Let \(B\) be a separable Banach space, and let \(F\) be a finite simple graph. A
\emph{\(B\)-decoration} of \(F\) is a map
\(f:E(F)\to B\). We write \(\mathbf F=(F,f)\) for the corresponding
\emph{\(B\)-decorated test graph}. If \(A\subset B\), then
\(\mathbf F\) is called \emph{\(A\)-decorated} if
\(f(e)\in A\) for every \(e\in E(F)\).

A function \(U:[0,1]^2\to B^*\) is called \emph{symmetric} if
\(U(x,y)=U(y,x)\) for almost every \((x,y)\in [0,1]^2\). It is called
\emph{weak-* measurable} if, for every \(b\in B\), the scalar function
\((x,y)\mapsto \langle b,U(x,y)\rangle\) is measurable. A
\emph{weak-* graphon} is a symmetric weak-* measurable function
\(U:[0,1]^2\to B^*\) such that
\[
\|U(\cdot,\cdot)\|_{B^*}\in L^p([0,1]^2)
\qquad
\text{for every }1\le p<\infty.
\]
For such \(U\), we write
\[
\|U\|_p
:=
\left(
\int_{[0,1]^2}\|U(x,y)\|_{B^*}^p\,dx\,dy
\right)^{1/p}.
\]
Since \(B\) is separable, the norm function
\((x,y)\mapsto \|U(x,y)\|_{B^*}\) is measurable. We denote the class
of \(B^*\)-valued weak-* graphons by \(\mathcal W_{\mathrm{w}^*}(B^*)\).
We call \(U\in\mathcal W_{\mathrm{w}^*}(B^*)\) \emph{bounded} if
\((x,y)\mapsto\|U(x,y)\|_{B^*}\) is essentially bounded, and in this case
write \(\|U\|_\infty\) for its essential supremum.

Let \(\mathbf F=(F,f)\) be a \(B\)-decorated test graph and let
\(U\in\mathcal W_{\mathrm{w}^*}(B^*)\). The \emph{homomorphism density} of
\(\mathbf F\) in \(U\) is
\[
t(\mathbf F,U)
=
\int_{[0,1]^{V(F)}}
\prod_{ij\in E(F)}
\langle f(ij),U(x_i,x_j)\rangle
\prod_{i\in V(F)}dx_i.
\]
The integral is well defined. Indeed, if \(m:=|E(F)|\ge 1\), then the
integrand satisfies
\[
\left|
\prod_{ij\in E(F)}
\langle f(ij),U(x_i,x_j)\rangle
\right|
\le
\left(\prod_{ij\in E(F)}\|f(ij)\|_B\right)
\prod_{ij\in E(F)}\|U(x_i,x_j)\|_{B^*}.
\]
Hence H\"older's inequality applies because
\(\|U(\cdot,\cdot)\|_{B^*}\in L^m([0,1]^2)\).

A subset \(A\subset B\) is called \emph{generating} if
\(\overline{\operatorname{span}}A=B\). It is called a
\emph{countable generating set} if it is both countable and
generating. The role of \(A\) is to reduce the class of test
graphs to a countable class. Indeed, if \(A\) is countable,
then there are only countably many \(A\)-decorated test
graphs.

We now state Theorem 3.7 of Kunszenti-Kov\'acs--Lov\'asz--Szegedy
\cite{kls2022decorated} in the notation of this paper.
The theorem gives the following representation statement. If a sequence
of weak-* graphons is uniformly bounded in every \(L^p\)-norm and its
homomorphism densities converge on a countable generating class of test
graphs, then the homomorphism densities converge for every
\(B\)-decorated test graph, and the limiting densities are represented by a
weak-* graphon.

\begin{theorem}
\label{thm:kls-3.7}
Let \(B\) be a separable Banach space, and let
\(A\subset B\) be a countable generating set. Let \((U_n)\) be a sequence of
weak-* graphons in \(\mathcal W_{\mathrm{w}^*}(B^*)\). Assume the following
two conditions.
\begin{enumerate}
\item For every \(1\le p<\infty\), there exists \(c_p<\infty\) such that
\(\|U_n\|_p\le c_p\) for every \(n\).
\item For every \(A\)-decorated test graph \(\mathbf F\), the sequence
\(\bigl(t(\mathbf F,U_n)\bigr)\) is convergent.
\end{enumerate}
Then there exists a weak-* graphon
\(U\in \mathcal W_{\mathrm{w}^*}(B^*)\) such that
\(t(\mathbf F,U_n)\to t(\mathbf F,U)\) for every \(B\)-decorated test graph
\(\mathbf F\). Moreover, \(U\) can be chosen so that
\(\|U\|_p\le c_p\) for every \(1\le p<\infty\).
\end{theorem}

\subsection{Bochner graphons}
\label{subsec:dual-finite-graph-framework}

An \emph{\(X\)-decorated target graph} is a pair
\(\mathbf G=(G,g)\), where \(G\) is a finite graph that is complete and
completely looped, and
\(g:E(G)\to X\) is a decoration map.
An \emph{\(X^*\)-decorated test graph} is a pair
\(\mathbf F=(F,f)\), where \(F\) is a finite simple graph (without loops) and
\(f:E(F)\to X^*\). Thus target decorations lie in \(X\), while test
decorations lie in \(X^*\). The target graph is taken to be complete and
completely looped so that every vertex map from a test graph to \(G\) is
allowed.

Let \(\mathbf F=(F,f)\) be an \(X^*\)-decorated test graph and let
\(\mathbf G=(G,g)\) be an \(X\)-decorated target graph. For a vertex
map \(\varphi:V(F)\to V(G)\), and for an edge \(ij\in E(F)\), we write
\(\varphi(ij)\) for the edge of \(G\) joining \(\varphi(i)\) and
\(\varphi(j)\). This edge exists because \(G\) is complete and completely
looped. The \emph{homomorphism number} of \(\mathbf F\) in \(\mathbf G\) is
\[
\operatorname{hom}(\mathbf F,\mathbf G)
=
\sum_{\varphi:V(F)\to V(G)}
\prod_{ij\in E(F)}
\left\langle
f(ij),
g(\varphi(ij))
\right\rangle .
\]
The \emph{homomorphism density} of \(\mathbf F\) in \(\mathbf G\) is
\[
t(\mathbf F,\mathbf G)
=
\frac{\operatorname{hom}(\mathbf F,\mathbf G)}
{|V(G)|^{|V(F)|}}.
\]
Equivalently, if \(\varphi:V(F)\to V(G)\) is chosen uniformly at random,
then
\[
t(\mathbf F,\mathbf G)
=
\mathbb E_{\varphi}\left[
\prod_{ij\in E(F)}
\left\langle
f(ij),
g(\varphi(ij))
\right\rangle
\right].
\]

For \(1\le p<\infty\), define the \(L^p\)-norm of \(\mathbf G\) by
\[
\|\mathbf G\|_p
=
\left(
\frac{1}{|V(G)|^2}
\sum_{u,v\in V(G)}
\|g(uv)\|_X^p
\right)^{1/p}.
\]
For \(p=\infty\), define
\(\|\mathbf G\|_\infty=\max_{u,v\in V(G)}\|g(uv)\|_X\).

We next introduce the Bochner graphons which potentially represent density
limits. An \emph{\(X\)-valued Bochner graphon} is a symmetric Bochner
measurable function \(W:[0,1]^2\to X\) such that
\(\|W(\cdot,\cdot)\|_X\in L^p([0,1]^2)\) for every
\(1\le p<\infty\). For such \(W\), we write
\[
\|W\|_p:=\left(\int_{[0,1]^2}\|W(x,y)\|_X^p\,dx\,dy\right)^{1/p}.
\]
We denote the class of \(X\)-valued Bochner graphons by
\(\mathcal W_{\mathrm{Boch}}(X)\).
We call \(W\in\mathcal W_{\mathrm{Boch}}(X)\) \emph{bounded} if
\((x,y)\mapsto\|W(x,y)\|_X\) is essentially bounded, and in this case
write \(\|W\|_\infty\) for its essential supremum.

Let \(\mathbf F=(F,f)\) be an \(X^*\)-decorated test graph and let
\(W\in\mathcal W_{\mathrm{Boch}}(X)\). The \emph{homomorphism density} of
\(\mathbf F\) in \(W\) is
\[
t(\mathbf F,W)
=
\int_{[0,1]^{V(F)}}
\prod_{ij\in E(F)}
\left\langle
f(ij),
W(x_i,x_j)
\right\rangle
\prod_{i\in V(F)}dx_i.
\]
As in the weak-* graphon case, this integral is well defined.
Every \(X\)-decorated target graph has an associated equal-step
graphon. Let \(\mathbf G=(G,g)\), and partition \([0,1]\) into intervals
\((I_u)_{u\in V(G)}\) of equal measure \(1/|V(G)|\). Define
\(W_{\mathbf G}(x,y)=g(uv)\) if \(x\in I_u\) and \(y\in I_v\). Then
\(W_{\mathbf G}\) is an \(X\)-valued Bochner graphon, and
\(\|W_{\mathbf G}\|_p=\|\mathbf G\|_p\) for every
\(1\le p\le\infty\).
Moreover, \(t(\mathbf F,W_{\mathbf G})=t(\mathbf F,\mathbf G)\) for every
\(X^*\)-decorated test graph \(\mathbf F\).

Motivated by Theorem~\ref{thm:kls-3.7}, we now define the graphon
representation property of \(X\). It asks whether the density limits of uniformly
\(L^p\)-bounded target graph sequences whose \(X^*\)-decorated homomorphism
densities converge are always represented by an \(X\)-valued Bochner
graphon.

\begin{definition}
\label{def:grp}
Let \(X\) be a Banach space. We say that \(X\) has the \emph{graphon
representation property} if the following holds: if a sequence
\((\mathbf G_n)\) of \(X\)-decorated
target graphs satisfies
\begin{enumerate}
\item \(\sup_n\|\mathbf G_n\|_p<\infty\) for every \(1\le p<\infty\);
\item for every \(X^*\)-decorated test graph \(\mathbf F\), the
sequence \((t(\mathbf F,\mathbf G_n))\) converges,
\end{enumerate}
then there exists an \(X\)-valued Bochner graphon
\(W\in\mathcal W_{\mathrm{Boch}}(X)\) such that
\[
t(\mathbf F,\mathbf G_n)\to t(\mathbf F,W)
\]
for every
\(X^*\)-decorated test graph \(\mathbf F\).
\end{definition}

We do not require \(|V(G_n)|\to\infty\). This causes no loss of generality:
one may replace every vertex of \(G_n\) by the same number of copies and
decorate the edge between a copy of \(u\) and a copy of \(v\) by the original
decoration \(g_n(uv)\). This replacement preserves all
homomorphism densities and all the norms defined above. Taking \(n\) copies of
every vertex therefore produces an equivalent sequence whose numbers of
vertices tend to infinity.

In Section~\ref{sec:bounded-representation}, we also introduce a bounded variant
by replacing condition~(1) in Definition~\ref{def:grp} with
\(\sup_n\|\mathbf G_n\|_\infty<\infty\) and requiring the representing
graphon to be bounded.

\begin{remark}
The requirement that the representing graphon be Bochner measurable is
natural from the viewpoint of finite graphs. Indeed, the equal-step graphon
associated with each \(X\)-decorated target graph takes only finitely many
values and is therefore Bochner measurable. This motivates seeking a graphon
representation of the limiting densities within the class of Bochner
measurable symmetric functions, rather than in the larger class of weakly
measurable symmetric functions.
\end{remark}

\section{Necessary conditions for graphon representability}
\label{sec:necessary-conditions}

This section derives weak sequential completeness and the Radon--Nikodým
property as necessary conditions for the graphon representation property.
Weak sequential completeness is detected by a sequence of one-vertex looped
graphs associated with a weakly Cauchy sequence. The Radon--Nikodým property
requires a more substantial argument: starting from a bounded operator
\(T:L^1([0,1])\to X\), we construct a convergent graph sequence and use
star test graphs to construct a representing density from a limiting
graphon. We first carry this out for separable Banach spaces. An argument using
a four-cycle test graph then yields the separable determination needed for the
general case.

\subsection{Weak sequential completeness}

Recall that a sequence \((x_n)\) in a Banach space \(X\) is \emph{weakly Cauchy} if,
for every \(\varphi\in X^*\), the real sequence \((\varphi(x_n))\) is
Cauchy. The space \(X\) is \emph{weakly sequentially complete} if every
weakly Cauchy sequence in \(X\) has a weak limit in \(X\)
\cite[p.~232]{aliprantisBurkinshaw2006positive}.

\begin{proposition}
\label{prop:grp-implies-wsc}
If \(X\) has the graphon representation property, then \(X\) is weakly
sequentially complete.
\end{proposition}

\begin{proof}
Let \((x_n)\) be a weakly Cauchy sequence in \(X\). We show that \((x_n)\)
has a weak limit in \(X\). For each \(n\), let \(\mathbf G_n=(G_n,g_n)\) be
the target graph with one vertex and one loop, where the loop is decorated by
\(x_n\). Then, by definition, \(\|\mathbf G_n\|_\infty=\|x_n\|\). By the
uniform boundedness principle, \((x_n)\) is norm bounded. Hence
\((\mathbf G_n)\) is uniformly \(L^\infty\)-bounded and therefore uniformly
\(L^p\)-bounded for every \(1\le p<\infty\).

Let \(\mathbf F=(F,f)\) be an \(X^*\)-decorated test graph. Since \(G_n\) has
only one vertex,
\(t(\mathbf F,\mathbf G_n)=\operatorname{hom}(\mathbf F,\mathbf G_n)
=\prod_{ij\in E(F)}\langle f(ij),x_n\rangle\).
Since \((x_n)\) is weakly Cauchy, each real sequence
\((\langle f(ij),x_n\rangle)\) converges. Thus \(t(\mathbf F,\mathbf G_n)\)
converges for every \(X^*\)-decorated test graph \(\mathbf F\). By the graphon
representation property, there is an \(X\)-valued Bochner graphon \(W\) such
that \(t(\mathbf F,\mathbf G_n)\to t(\mathbf F,W)\) for every
\(X^*\)-decorated test graph \(\mathbf F\).

Fix \(\varphi\in X^*\), and let \(\mathbf F_\varphi\) be the test graph with
two vertices and one edge decorated by \(\varphi\). Then
\(t(\mathbf F_\varphi,\mathbf G_n)=\langle\varphi,x_n\rangle\).
On the other hand,
\(t(\mathbf F_\varphi,W)
=\int_{[0,1]^2}\langle\varphi,W(x_1,x_2)\rangle\,dx_1dx_2
=\langle\varphi,\int_{[0,1]^2}W(x_1,x_2)\,dx_1dx_2\rangle\).
Therefore
\(\langle\varphi,x_n\rangle
\to\langle\varphi,\int_{[0,1]^2}W(x_1,x_2)\,dx_1dx_2\rangle\)
for every \(\varphi\in X^*\). Hence \((x_n)\) converges weakly in \(X\) to
\(\int_{[0,1]^2}W(x_1,x_2)\,dx_1dx_2\). Thus \(X\) is weakly sequentially
complete.
\end{proof}

\subsection{The Radon--Nikod\'ym property: the separable case}
\label{subsec:rnp-separable-case}

We next recall the Radon--Nikod\'ym property. Let \(X\) be a Banach space and
let \((\Omega,\Sigma,\mu)\) be a finite measure space. We say that \(X\) has
the \emph{Radon--Nikod\'ym property with respect to
\((\Omega,\Sigma,\mu)\)} if every countably additive \(X\)-valued measure
\(\nu:\Sigma\to X\) of bounded variation such that \(\nu\ll\mu\) has a
Radon--Nikod\'ym derivative. That is, there exists \(g\in L^1(\Omega;X)\) such
that \(\nu(A)=\int_A g\,d\mu\) for every \(A\in\Sigma\).
We say that \(X\) has the \emph{Radon--Nikod\'ym property} if it has the
Radon--Nikod\'ym property with respect to every finite measure space.

We shall use the following standard characterization of the Radon--Nikod\'ym
property by representable operators
\cite[Theorem~III.1.5]{diestelUhl1977vector}.

\begin{theorem}
\label{thm:operator-characterization-rnp}
Let \((\Omega,\Sigma,\mu)\) be a finite measure space and let \(X\) be a
Banach space. Then \(X\) has the Radon--Nikod\'ym property with respect to
\((\Omega,\Sigma,\mu)\) if and only if every bounded operator
\(T:L^1(\Omega)\to X\) is representable, i.e. there exists
\(g\in L^\infty(\Omega;X)\) such that
\(T(f)=\int_\Omega f g\,d\mu\) for every \(f\in L^1(\Omega)\).
\end{theorem}

We shall also use the following testing theorem for the Lebesgue interval
\cite[Corollary~V.3.8]{diestelUhl1977vector}; see also
\cite[Corollary~2.14]{pisier2016martingales}.

\begin{theorem}
\label{thm:interval-testing-rnp}
Let \(X\) be a Banach space. Then \(X\) has the Radon--Nikod\'ym property if
and only if \(X\) has the Radon--Nikod\'ym property with respect to the
Lebesgue unit interval \([0,1]\).
\end{theorem}

It therefore remains to show that every bounded operator
\(T:L^1([0,1])\to X\) admits a representing density. The next lemma isolates
the graph construction associated with such an operator.

\begin{lemma}
\label{lem:dyadic-finite-graphs-operator}
Let \(X\) be a Banach space, and let \(T:L^1([0,1])\to X\) be a bounded
operator. For each \(n\ge1\) and \(i=0,1,\dots,2^n-1\), put
\(I_{n,i}=[i/2^n,(i+1)/2^n)\).
Then there exists a sequence \((\mathbf G_n)_{n\ge1}\) of \(X\)-decorated
target graphs \(\mathbf G_n=(G_n,a_n)\), where \(G_n\) is the complete
looped graph with vertex set \(\{0,1,\dots,2^n-1\}\), such that
\begin{enumerate}
\item \((\mathbf G_n)\) is uniformly \(L^\infty\)-bounded.
\item \(t(\mathbf F,\mathbf G_n)\) converges for every
\(X^*\)-decorated test graph \(\mathbf F\).
\item For every \(i=0,1,\dots,2^n-1\), we have
\(2^{-n}\sum_{j=0}^{2^n-1}a_n(ij)=2^nT(1_{I_{n,i}})\).
\end{enumerate}
\end{lemma}

\begin{proof}
For each \(n\), let \(G_n\) be the complete, completely looped graph with vertex set
\(\{0,1,\dots,2^n-1\}\), and define the edge decoration \(a_n:E(G_n)\to X\)
by
\(a_n(ij)=2^nT(1_{I_{n,i}})+2^nT(1_{I_{n,j}})-T1\) for
\(0\le i,j<2^n\).
Set \(\mathbf G_n=(G_n,a_n)\).

We first check the uniform boundedness in (1). Since
\[
\|2^nT(1_{I_{n,i}})\|\le 2^n\|T\|\,|I_{n,i}|=\|T\|,
\qquad \|T1\|\le\|T\|,
\]
we have \(\|a_n(ij)\|\le 3\|T\|\) for all \(n,i,j\).
Hence \((\mathbf G_n)\) is uniformly \(L^\infty\)-bounded.

Next we verify (3) by computing the degree average. For every \(i\),
\[
2^{-n}\sum_{j=0}^{2^n-1}a_n(ij)
=2^nT(1_{I_{n,i}})+2^{-n}\sum_{j=0}^{2^n-1}2^nT(1_{I_{n,j}})-T1.
\]
Since \(2^{-n}\sum_{j=0}^{2^n-1}2^nT(1_{I_{n,j}})
=\sum_{j=0}^{2^n-1}T(1_{I_{n,j}})=T1\), we get
\(2^{-n}\sum_{j=0}^{2^n-1}a_n(ij)=2^nT(1_{I_{n,i}})\).

It remains to prove the convergence in (2). Let \(W_n\) be
the equal-step graphon associated with \(\mathbf G_n\). Fix
\(\varphi\in X^*\). Since \(\varphi\circ T\) is a bounded linear functional on
\(L^1([0,1])\), by the Riesz representation theorem, there is
\(h_\varphi\in L^\infty([0,1])\) such that
\(\langle\varphi,Tu\rangle=\int_0^1u(x)h_\varphi(x)\,dx\) for every
\(u\in L^1([0,1])\).

Let \(x\in I_{n,i}\) and \(y\in I_{n,j}\). By the definitions of \(W_n\) and
\(a_n\),
\(\langle\varphi,W_n(x,y)\rangle=\langle\varphi,a_n(ij)\rangle
=2^n\langle\varphi,T1_{I_{n,i}}\rangle
+2^n\langle\varphi,T1_{I_{n,j}}\rangle-\langle\varphi,T1\rangle\).
Using the representation by \(h_\varphi\), we obtain
\(\langle\varphi,W_n(x,y)\rangle
=\frac{1}{|I_{n,i}|}\int_{I_{n,i}}h_\varphi(s)\,ds
+\frac{1}{|I_{n,j}|}\int_{I_{n,j}}h_\varphi(s)\,ds
-\int_0^1h_\varphi(s)\,ds\).
In other words, if \(\mathcal D_n\) denotes the sub-\(\sigma\)-algebra
generated by the intervals \((I_{n,k})_k\), then
\(\langle\varphi,W_n(x,y)\rangle
=\mathbb E(h_\varphi\mid\mathcal D_n)(x)
+\mathbb E(h_\varphi\mid\mathcal D_n)(y)-\int_0^1h_\varphi(s)\,ds\).

The sequence \((\mathcal D_n)_n\) generates the Borel \(\sigma\)-algebra on
\([0,1]\). Hence, by the martingale convergence theorem,
\(\mathbb E(h_\varphi\mid \mathcal D_n)(x)\to h_\varphi(x)\) for a.e. \(x\).
Therefore \(\langle \varphi,W_n\rangle\) converges a.e. on \([0,1]^2\). Since
\((\mathbf G_n)\) is uniformly \(L^\infty\)-bounded,
\(\langle \varphi,W_n\rangle\) is also uniformly bounded in \(n\).

Now let \(\mathbf F=(F,f)\) be an \(X^*\)-decorated test graph.
For each edge \(uv\in E(F)\), applying the preceding paragraph to
\(\varphi=f(uv)\) shows that \(\langle f(uv),W_n(x_u,x_v)\rangle\)
converges for a.e. \((x_v)_{v\in V(F)}\in[0,1]^{V(F)}\) and is uniformly
bounded in \(n\). Since \(F\) has only finitely many edges, the product
\[
\prod_{uv\in E(F)}\langle f(uv),W_n(x_u,x_v)\rangle
\]
also converges a.e. on \([0,1]^{V(F)}\) and is uniformly bounded in \(n\).
Hence, by the dominated convergence theorem, \(t(\mathbf F,W_n)\), which is
the integral
\[
\int_{[0,1]^{V(F)}}\prod_{uv\in E(F)}
\langle f(uv),W_n(x_u,x_v)\rangle
\prod_{i\in V(F)} dx_i
\]
converges as \(n\to\infty\). Since \(W_n\) is the equal-step graphon associated
with \(\mathbf G_n\), this means that \(t(\mathbf F,\mathbf G_n)\) converges
as \(n\to\infty\).
\end{proof}

We now apply the graphon representation property to this graph sequence to
obtain a limiting graphon. Star test graphs then yield the identities needed
to construct a representing density.

\begin{lemma}
\label{lem:star-graph-moment-identities}
Let \(X\) be a Banach space with the graphon representation property, and let
\(T:L^1([0,1])\to X\) be a bounded operator. For each \(\varphi\in X^*\), let
\(h_\varphi\in L^\infty([0,1])\) be defined by
\(\langle\varphi,Tu\rangle=\int_0^1u(x)h_\varphi(x)\,dx\) for every
\(u\in L^1([0,1])\).
Then there exists an \(X\)-valued graphon \(W\) such that, for every
\(k\ge1\) and every \(\varphi_1,\dots,\varphi_k\in X^*\) (with repetitions
allowed),
\[
\int_0^1\prod_{r=1}^k\langle\varphi_r,D_W(x)\rangle\,dx
=\int_0^1\prod_{r=1}^k h_{\varphi_r}(x)\,dx,
\]
where \(D_W(x):=\int_0^1 W(x,y)\,dy\) for a.e. \(x\).
\end{lemma}

\begin{proof}
Applying Lemma~\ref{lem:dyadic-finite-graphs-operator} to \(T\), we obtain a sequence
\((\mathbf G_n)_{n\ge1}\) of \(X\)-decorated graphs.
By that lemma, this sequence is uniformly \(L^\infty\)-bounded, and its
densities converge against every \(X^*\)-decorated test graph.
By the graphon representation property, there exists an \(X\)-valued graphon
\(W\) such that \(t(\mathbf F,\mathbf G_n)\to t(\mathbf F,W)\) for every
\(X^*\)-decorated test graph \(\mathbf F\). Put
\(D_W(x):=\int_0^1 W(x,y)\,dy\) for a.e. \(x\).

Fix \(k\ge1\) and \(\varphi_1,\dots,\varphi_k\in X^*\). Let
\(\mathbf S_k\) be the star graph with center \(0\) and leaves \(1,\dots,k\),
where the edge joining \(0\) and \(r\) is decorated by \(\varphi_r\)
(see Figure~\ref{fig:star-test-graph}).
\begin{figure}
\centering
\begin{tikzpicture}[
    vertex/.style={circle, draw, inner sep=1.5pt, minimum size=6mm},
    edge label/.style={midway, fill=white, inner sep=1pt}
]

\node[vertex] (c) at (0,0) {$0$};
\node[vertex] (v1) at (2,1.6) {$1$};
\node[vertex] (v2) at (2,0.8) {$2$};
\node[vertex] (v3) at (2,0) {$3$};
\node at (2,-0.6) {$\vdots$};
\node[vertex] (vk) at (2,-1.6) {$k$};

\draw (c) -- (v1) node[edge label] {$\varphi_1$};
\draw (c) -- (v2) node[edge label] {$\varphi_2$};
\draw (c) -- (v3) node[edge label] {$\varphi_3$};
\draw (c) -- (vk) node[edge label] {$\varphi_k$};

\node[left=4pt of c] {center};
\node[right=4pt of v1] {leaf};
\node[right=4pt of v2] {leaf};
\node[right=4pt of v3] {leaf};
\node[right=4pt of vk] {leaf};

\end{tikzpicture}
\caption{The star test graph \(\mathbf S_k\). The edge joining the center
\(0\) to the leaf \(r\) is decorated by \(\varphi_r\).}
\label{fig:star-test-graph}
\end{figure}
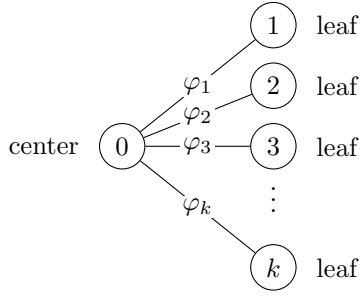
First we compute the density of \(\mathbf S_k\) in \(W\). By the definition of the star test graph,
\(t(\mathbf S_k,W)=\int_{[0,1]^{k+1}}
\prod_{r=1}^k\langle\varphi_r,W(x_0,x_r)\rangle
\,dx_0dx_1\cdots dx_k\).
For fixed \(x_0\), the factor indexed by \(r\) depends only on the leaf
variable \(x_r\). Thus Fubini's theorem gives
\(t(\mathbf S_k,W)=\int_0^1\prod_{r=1}^k
(\int_0^1\langle\varphi_r,W(x_0,x_r)\rangle\,dx_r)\,dx_0\).
Since
\[
\int_0^1\langle\varphi_r,W(x_0,x_r)\rangle\,dx_r
=\langle\varphi_r,\int_0^1W(x_0,y)\,dy\rangle
=\langle\varphi_r,D_W(x_0)\rangle,
\]
we obtain
\begin{equation}\label{eq:star-density-W}
t(\mathbf S_k,W)=\int_0^1
\prod_{r=1}^k\langle\varphi_r,D_W(x_0)\rangle\,dx_0.
\end{equation}

We next compute the same density in \(\mathbf G_n\). The vertices of
\(\mathbf S_k\) are the center \(0\) and the leaves \(1,\dots,k\). Hence a
map from \(V(\mathbf S_k)\) to \(V(\mathbf G_n)\) is determined by the image
\(i_0\) of the center and the images \(i_1,i_2,\dots,i_k\) of the leaves
(\(i_r\in\{0,1,\dots,2^n-1\}\)). Since
\(|V(\mathbf G_n)|=2^n\) and \(|V(\mathbf S_k)|=k+1\), the definition of
homomorphism density gives
\(t(\mathbf S_k,\mathbf G_n)=2^{-n(k+1)}
\sum_{i_0,\dots,i_k=0}^{2^n-1}
\prod_{r=1}^k\langle\varphi_r,a_n(i_0i_r)\rangle\).
This can be rewritten as
\[
\begin{aligned}
t(\mathbf S_k,\mathbf G_n)
&=2^{-n}\sum_{i=0}^{2^n-1}
\prod_{r=1}^k\left(2^{-n}\sum_{j=0}^{2^n-1}
\langle\varphi_r,a_n(ij)\rangle\right)\\
&=2^{-n}\sum_{i=0}^{2^n-1}
\prod_{r=1}^k\left\langle
\varphi_r,2^{-n}\sum_{j=0}^{2^n-1}a_n(ij)
\right\rangle\\
&=2^{-n}\sum_{i=0}^{2^n-1}
\prod_{r=1}^k\langle\varphi_r,2^nT(1_{I_{n,i}})\rangle.
\end{aligned}
\]
The first equality follows by rearranging the finite sums. Since each
\(\varphi_r\) is linear,
\[
2^{-n}\sum_{j=0}^{2^n-1}
\langle\varphi_r,a_n(ij)\rangle
=
\left\langle
\varphi_r,2^{-n}\sum_{j=0}^{2^n-1}a_n(ij)
\right\rangle,
\]
which gives the second equality. The last equality follows from
(3) in Lemma~\ref{lem:dyadic-finite-graphs-operator}.

On the other hand, let \(\mathcal D_n\) denote the sub-\(\sigma\)-algebra
generated by the intervals \((I_{n,k})_k\). For \(x\in I_{n,i}\), we have
\(\mathbb E(h_{\varphi_r}\mid\mathcal D_n)(x)
=\frac{1}{|I_{n,i}|}\int_{I_{n,i}}h_{\varphi_r}(s)\,ds
=2^n\int_{I_{n,i}}h_{\varphi_r}(s)\,ds\).
Since \(h_{\varphi_r}\) represents \(\varphi_r\circ T\), this is equal to
\(2^n\langle\varphi_r,T1_{I_{n,i}}\rangle
=\langle\varphi_r,2^nT(1_{I_{n,i}})\rangle\).
Hence
\[
\int_0^1\prod_{r=1}^k\mathbb E(h_{\varphi_r}\mid\mathcal D_n)(x)\,dx
=2^{-n}\sum_{i=0}^{2^n-1}
\prod_{r=1}^k\langle\varphi_r,2^nT(1_{I_{n,i}})\rangle.
\]
Combining this with the previous computation, we obtain
\[
t(\mathbf S_k,\mathbf G_n)=\int_0^1
\prod_{r=1}^k\mathbb E(h_{\varphi_r}\mid\mathcal D_n)(x)\,dx.
\]

The sequence \((\mathcal D_n)_n\) generates the Borel \(\sigma\)-algebra on
\([0,1]\). Hence, by the martingale convergence theorem,
\(\mathbb E(h_{\varphi_r}\mid \mathcal D_n)(x)\to h_{\varphi_r}(x)\) for
a.e. \(x\). Moreover, these conditional expectations are uniformly bounded.
Therefore, by the dominated convergence theorem,
\(t(\mathbf S_k,\mathbf G_n)=\int_0^1
\prod_{r=1}^k\mathbb E(h_{\varphi_r}\mid\mathcal D_n)(x)\,dx
\to\int_0^1\prod_{r=1}^k h_{\varphi_r}(x)\,dx\).

By the choice of \(W\),
\(t(\mathbf S_k,\mathbf G_n)\to t(\mathbf S_k,W)\). Combining this with
\eqref{eq:star-density-W}, we obtain
\(\int_0^1\prod_{r=1}^k\langle\varphi_r,D_W(x)\rangle\,dx
=\int_0^1\prod_{r=1}^k h_{\varphi_r}(x)\,dx\).
This proves the lemma.
\end{proof}

We can now prove the result for separable Banach spaces.

\begin{theorem}
\label{thm:separable-grp-implies-rnp}
Let \(X\) be a separable Banach space. If \(X\) has the graphon representation
property, then \(X\) has the Radon--Nikod\'ym property.
\end{theorem}

\begin{proof}
By Theorem~\ref{thm:interval-testing-rnp}, it is enough to prove that \(X\)
has the Radon--Nikod\'ym property with respect to the Lebesgue unit interval.
By Theorem~\ref{thm:operator-characterization-rnp}, this amounts to proving
that every bounded operator \(T:L^1([0,1])\to X\) is representable. Fix such an
operator \(T\).

For each \(\varphi\in X^*\), the linear functional \(\varphi\circ T\) belongs
to \((L^1([0,1]))^*\). By the Riesz representation theorem, there exists
\(h_\varphi\in L^\infty([0,1])\) such that
\(\langle \varphi,Tu\rangle=\int_0^1u(x)h_\varphi(x)\,dx\) for every
\(u\in L^1([0,1])\).
Moreover \(\|h_\varphi\|_{L^\infty}=\|\varphi\circ T\|\le
\|\varphi\|\,\|T\|\).

By Lemma~\ref{lem:star-graph-moment-identities}, there exists an \(X\)-valued
Bochner graphon \(W\) such that, for every \(k\ge1\) and every
\(\varphi_1,\dots,\varphi_k\in X^*\),
\begin{equation}
\tag{*}
\label{eq:separable-rnp-star-moments}
\int_0^1
\prod_{r=1}^k \langle \varphi_r,D(x)\rangle\,dx
=
\int_0^1
\prod_{r=1}^k h_{\varphi_r}(x)\,dx,
\end{equation}
where \(D(x):=D_W(x)=\int_0^1 W(x,y)\,dy\).

Since \(X\) is separable, the dual unit ball \(B_{X^*}\) is weak-* compact and
metrizable. Choose a countable weak-* dense subset
\((\psi_j)_{j\ge1}\subset B_{X^*}\). This family norms \(X\): for every
\(x\in X\), \(\|x\|=\sup_{j\ge1}|\langle \psi_j,x\rangle|\).
Put \(h_j:=h_{\psi_j}\). Since \(\|h_j\|_{L^\infty}\le \|T\|\), and since
only countably many functions are involved, we may choose representatives so
that \(|h_j(x)|\le\|T\|\) for every \(x\in[0,1]\) and \(j\ge1\).

Define \(J:X\to\R^{\N}\) by
\(J(x):=(\langle \psi_j,x\rangle)_{j\ge1}\).
The family \((\psi_j)\) separates points of \(X\), so \(J\) is injective. Each
coordinate is continuous, hence \(J\) is Borel measurable. Since both \(X\) and
\(\R^{\N}\) are Polish spaces, the Lusin--Souslin theorem
\cite[Theorem~6.8.6]{bogachev2007measure} implies that
\(J(X)\subset\R^{\N}\) is Borel and that \(J^{-1}:J(X)\to X\) is Borel
measurable.

We now compare the following two \(\R^{\N}\)-valued measurable functions. Let
\(\mathbf h(x):=(h_j(x))_{j\ge1}\), and put \(\mathbf q:=J\circ D\). Thus
\(\mathbf q(x)=(\langle \psi_j,D(x)\rangle)_{j\ge1}\) for all \(x\).
Set \(K:=[-\|T\|,\|T\|]^{\N}\). By the choice of representatives,
\(\mathbf h(x)\in K\) for every \(x\). Fix \(j\ge1\). Applying
\eqref{eq:separable-rnp-star-moments} with
\(\psi_j\) repeated \(2m\) times gives
\(\int_0^1|\langle \psi_j,D(x)\rangle|^{2m}\,dx
=\int_0^1|h_j(x)|^{2m}\,dx\le\|T\|^{2m}\).
Let \(\lambda\) denote Lebesgue measure on \([0,1]\). Therefore, for every
\(a>\|T\|\),
\(\lambda\{x:|\langle \psi_j,D(x)\rangle|>a\}
\le(\|T\|/a)^{2m}\).
Letting \(m\to\infty\), we get
\(|\langle \psi_j,D(x)\rangle|\le \|T\|\) for a.e. \(x\). Since the set of
indices \(j\) is countable, by changing \(\mathbf q\) on a null set, we may
and do assume that \(\mathbf q(x)\in K\) for every \(x\).

We claim that \(\mathbf q\) and \(\mathbf h\) have the same distribution.
Write \(q_j\) for the \(j\)-th coordinate of \(\mathbf q\), and fix \(d\ge1\).
For a multi-index
\(\alpha=(\alpha_1,\dots,\alpha_d)\in\mathbb N^d\), applying
\eqref{eq:separable-rnp-star-moments} to the list in which \(\psi_j\) is
repeated \(\alpha_j\) times gives
\(\int_0^1\prod_{j=1}^d q_j(x)^{\alpha_j}\,dx
=\int_0^1\prod_{j=1}^d\langle\psi_j,D(x)\rangle^{\alpha_j}\,dx
=\int_0^1\prod_{j=1}^d h_j(x)^{\alpha_j}\,dx\).
Thus the \(K_d\)-valued measurable functions \((q_1,\dots,q_d)\) and
\((h_1,\dots,h_d)\), where \(K_d:=[-\|T\|,\|T\|]^d\), have the same polynomial
moments. By the Stone--Weierstrass theorem, polynomials are uniformly dense in
\(C(K_d)\), and hence these two measurable functions have the same distribution
on \(K_d\).

The finite-dimensional distributions agree for every \(d\). Since the Borel
\(\sigma\)-algebra of \(K\) is generated by finite-coordinate cylinders,
\(\mathbf q\) and \(\mathbf h\) have the same distribution on \(\R^{\N}\).
Since \(\mathbf q=J\circ D\) a.e., we have \(\mathbf q(x)\in J(X)\) for a.e. \(x\);
hence the same distribution implies that \(\mathbf h(x)\in J(X)\) for a.e.
\(x\).

On this full-measure set define \(g(x):=J^{-1}(\mathbf h(x))\), and define
\(g(x)=0\) outside it. Since \(J^{-1}:J(X)\to X\) is Borel measurable, \(g\) is
measurable. By the Pettis measurability theorem, and since \(X\) is separable,
\(g\) is Bochner measurable. Moreover, for a.e. \(x\),
\(\langle \psi_j,g(x)\rangle=h_j(x)\) for all \(j\).
Therefore
\(\|g(x)\|=\sup_{j\ge1}|\langle\psi_j,g(x)\rangle|
=\sup_{j\ge1}|h_j(x)|\le\|T\|\)
for a.e. \(x\). Hence \(g\in L^\infty([0,1];X)\). Let
\(u\in L^1([0,1])\). For
every \(j\ge1\),
\[
\begin{aligned}
\langle\psi_j,Tu\rangle
&=\int_0^1u(x)h_j(x)\,dx \\
&=\int_0^1u(x)\langle\psi_j,g(x)\rangle\,dx \\
&=\left\langle\psi_j,\int_0^1u(x)g(x)\,dx\right\rangle.
\end{aligned}
\]
Since \((\psi_j)\) separates points of \(X\), we conclude that
\(Tu=\int_0^1 u(x)g(x)\,dx\). Since this holds for arbitrary
\(u\in L^1([0,1])\), the operator \(T\) is representable.
\end{proof}

\subsection{Separable determination and the general case}
\label{subsec:separable-descent-general-rnp}

To establish separable determination, we use the following four-cycle lemma.
Although it follows from more general results
\cite[(E.22)--(E.23), p.~67]{janson2013graphons}, the form below is sufficient
for our purposes and admits a short direct proof, which we include for completeness.

\begin{lemma}
\label{lem:four-cycle-zero-kernel}
Let \(u\in L^2([0,1]^2)\) be real-valued and symmetric. Then the integral
\[
I
:=
\int_{[0,1]^4}
u(x_1,x_2)u(x_2,x_3)u(x_3,x_4)u(x_4,x_1)
\,dx_1\,dx_2\,dx_3\,dx_4
\]
is finite. If \(I=0\), then \(u=0\) a.e. on
\([0,1]^2\).
\end{lemma}

\begin{proof}
For a.e. \((y,z)\in[0,1]^2\), define
\(k(y,z):=\int_0^1u(x,y)u(x,z)\,dx\). For a.e. \((y,z)\),
\[
|k(y,z)|
\le
\left(\int_0^1|u(x,y)|^2\,dx\right)^{1/2}
\left(\int_0^1|u(x,z)|^2\,dx\right)^{1/2},
\]
and hence
\[
\|k\|_2^2
\le
\int_{[0,1]^2}
\|u(\cdot,y)\|_2^2\|u(\cdot,z)\|_2^2
\,dy\,dz
=
\|u\|_2^4.
\]
Thus \(k\in L^2([0,1]^2)\). The same estimate applied to \(|u|\) shows
that \(I\) is finite.

Using the symmetry of \(u\) and Fubini's theorem, we obtain
\[
I=\int_{[0,1]^2}k(y,z)^2\,dy\,dz.
\]
Thus \(I=0\) implies that \(k=0\) a.e.

For \(f\in L^2([0,1])\), define
\(g_f(x):=\int_0^1u(x,y)f(y)\,dy\). The inequality
\(|g_f(x)|\le\|u(x,\cdot)\|_2\|f\|_2\) shows that
\(g_f\in L^2([0,1])\). By Fubini's theorem,
\[
\begin{aligned}
\|g_f\|_2^2
&=
\int_{[0,1]^3}
u(x,y)u(x,z)f(y)f(z)
\,dx\,dy\,dz\\
&=
\int_{[0,1]^2}k(y,z)f(y)f(z)\,dy\,dz
=0.
\end{aligned}
\]
Therefore \(g_f=0\) a.e. for every \(f\in L^2([0,1])\).

Choose a countable dense subset \(D\subset L^2([0,1])\). Since
\(D\) is countable, there is a null set \(N\subset[0,1]\) such
that, for every \(x\notin N\), the function \(u(x,\cdot)\) belongs to
\(L^2([0,1])\) and \(\int_0^1u(x,y)f(y)\,dy=0\) for every
\(f\in D\).
By density, \(u(x,\cdot)\) is orthogonal to every element of \(L^2([0,1])\),
and hence \(u(x,\cdot)=0\) in \(L^2([0,1])\). This holds for every
\(x\notin N\), so
\(\|u\|_2^2=\int_0^1\|u(x,\cdot)\|_2^2\,dx=0\). Therefore \(u=0\)
a.e. on \([0,1]^2\).
\end{proof}

\begin{theorem}[Separable determination]
\label{thm:separable-descent}
The graphon representation property is separably determined; that is, a
Banach space \(X\) has the graphon representation property if and only if
every closed separable subspace of \(X\) has the graphon representation
property.
\end{theorem}

\begin{proof}
First suppose that \(X\) has the graphon representation property. Let
\(Y\subset X\) be a closed separable subspace. We prove that \(Y\) has the
graphon representation property.

Let \((\mathbf G_n)\) be a sequence of \(Y\)-decorated target graphs
satisfying the two hypotheses in Definition~\ref{def:grp}. Regard the same
sequence as \(X\)-decorated. If \(\mathbf F=(F,f)\) is an
\(X^*\)-decorated test graph, let \(\mathbf F|_Y\) denote the
\(Y^*\)-decorated test graph obtained by restricting every edge decoration to
\(Y\). Since the decorations of \(\mathbf G_n\) lie in \(Y\), we have
\(t(\mathbf F,\mathbf G_n)=t(\mathbf F|_Y,\mathbf G_n)\). Thus the convergence
of all \(Y^*\)-decorated densities implies the convergence of all
\(X^*\)-decorated densities, while the norm bounds are unchanged. Hence, by
the graphon representation property of \(X\), there is an \(X\)-valued
Bochner graphon \(W\) such that
\(t(\mathbf F,\mathbf G_n)\to t(\mathbf F,W)\) for every
\(X^*\)-decorated test graph
\(\mathbf F\).

We show that \(W\) is \(Y\)-valued a.e. Since \(W\) is Bochner measurable,
it has essentially separable range. Since \(Y\) is separable, there is a
closed separable subspace \(Z\subset X\) such that \(Y\subset Z\) and
\(W(s,t)\in Z\) for a.e. \((s,t)\). Let
\(K:=\{z^*\in B_{Z^*}:z^*|_Y=0\}\). Since \(Z\) is separable,
\(B_{Z^*}\) is weak-* compact and metrizable, and hence \(K\) has a countable
weak-* dense subset \(K_0\).

Fix \(z^*\in K_0\). By the Hahn--Banach theorem, there exists an extension
\(\widetilde z^*\in X^*\) of \(z^*\). Since \(Y\subset Z\) and
\(z^*|_Y=0\), the extension \(\widetilde z^*\) vanishes on \(Y\). Let
\(C_4^{\widetilde z^*}\) denote the \(X^*\)-decorated test graph obtained by
decorating every edge of the four-vertex cycle \(C_4\) by
\(\widetilde z^*\). Since each \(\mathbf G_n\) is
\(Y\)-decorated, \(t(C_4^{\widetilde z^*},\mathbf G_n)=0\) for every \(n\).
The density \(t(C_4^{\widetilde z^*},W)\) is the limit of this sequence, and
is therefore zero. Put
\(u(s,t):=\langle\widetilde z^*,W(s,t)\rangle\). Then \(u\) is a real-valued
symmetric function in \(L^2([0,1]^2)\). By definition, the density
\(t(C_4^{\widetilde z^*},W)\) is precisely the fourfold integral in
Lemma~\ref{lem:four-cycle-zero-kernel}. That lemma therefore gives
\(\langle z^*,W(s,t)\rangle
=\langle\widetilde z^*,W(s,t)\rangle=u(s,t)=0\) for a.e. \((s,t)\).

Since \(K_0\) is countable, outside a null set we have \(W(s,t)\in Z\) and
\(\langle z^*,W(s,t)\rangle=0\) for every \(z^*\in K_0\). Fix \((s,t)\)
outside this null set. The map
\(K\ni z^*\mapsto\langle z^*,W(s,t)\rangle\) is continuous with respect to
the weak-* topology and vanishes on the dense subset \(K_0\). Hence it
vanishes on all of \(K\). By the Hahn--Banach theorem, this implies that
\(W(s,t)\in Y\). Thus \(W\) is \(Y\)-valued a.e. and may be regarded as a
\(Y\)-valued Bochner graphon.

Finally, let \(\mathbf F=(F,f)\) be a \(Y^*\)-decorated test graph. For each
\(ij\in E(F)\), use the Hahn--Banach theorem to choose an extension
\(\widetilde f(ij)\in X^*\) of \(f(ij)\), and put
\(\widetilde{\mathbf F}:=(F,\widetilde f)\). Since the decorations of
\(\mathbf G_n\) lie in \(Y\), the definition of density gives
\(t(\mathbf F,\mathbf G_n)=t(\widetilde{\mathbf F},\mathbf G_n)\) for every
\(n\). The convergence for \(X^*\)-decorated test graphs gives
\(t(\widetilde{\mathbf F},\mathbf G_n)\to
t(\widetilde{\mathbf F},W)\). Since \(W\) is \(Y\)-valued, the definition of
density also gives \(t(\widetilde{\mathbf F},W)=t(\mathbf F,W)\).
Consequently, \(t(\mathbf F,\mathbf G_n)\to t(\mathbf F,W)\), and hence \(Y\)
has the graphon representation property.

Conversely, suppose that every closed separable subspace of \(X\) has the
graphon representation property. Let \((\mathbf G_n)\) be a sequence of
\(X\)-decorated target graphs satisfying the two hypotheses in
Definition~\ref{def:grp}. The set of all decorations appearing in the sequence
\((\mathbf G_n)\) is countable. Let \(Y\) be the closed linear span of these
decorations. Then \(Y\) is a closed separable subspace of \(X\), and each
\(\mathbf G_n\) is \(Y\)-decorated.

By assumption, \(Y\) has the graphon representation property. For a
\(Y^*\)-decorated test graph \(\mathbf F\), use the same construction as above
to obtain an \(X^*\)-decorated test graph \(\widetilde{\mathbf F}\). Then
\(t(\mathbf F,\mathbf G_n)=t(\widetilde{\mathbf F},\mathbf G_n)\), and the
latter sequence converges by hypothesis. Thus
\(t(\mathbf F,\mathbf G_n)\) converges for every \(Y^*\)-decorated test graph
\(\mathbf F\). The norm bounds are unchanged, so \((\mathbf G_n)\) satisfies
the two hypotheses in the definition of the graphon representation property
for \(Y\). Hence there is a \(Y\)-valued Bochner graphon \(W\) such that
\(t(\mathbf F,\mathbf G_n)\to t(\mathbf F,W)\) for every
\(Y^*\)-decorated test graph \(\mathbf F\).

Regard \(W\) as an \(X\)-valued graphon, and let \(\mathbf F=(F,f)\) be an
\(X^*\)-decorated test graph. Let \(\mathbf F|_Y\) be the
\(Y^*\)-decorated test graph obtained by restricting every edge decoration of
\(\mathbf F\) to \(Y\). Since the decorations of \(\mathbf G_n\) and the
values of \(W\) lie in \(Y\), the definition of density gives
\(t(\mathbf F,\mathbf G_n)=t(\mathbf F|_Y,\mathbf G_n)\) and
\(t(\mathbf F,W)=t(\mathbf F|_Y,W)\). The density convergence for
\(Y^*\)-decorated test graphs applied to \(\mathbf F|_Y\) therefore gives
\(t(\mathbf F,\mathbf G_n)\to t(\mathbf F,W)\). Thus \(X\) has the graphon
representation property.
\end{proof}

We now obtain the general Radon--Nikod\'ym consequence.

\begin{corollary}
\label{cor:grp-implies-rnp}
Let \(X\) be a Banach space. If \(X\) has the graphon representation property,
then \(X\) has the Radon--Nikod\'ym property.
\end{corollary}

\begin{proof}
Assume that \(X\) has the graphon representation property. Let \(Y\subset X\)
be a closed separable subspace. By Theorem~\ref{thm:separable-descent}, \(Y\)
has the graphon representation property. Since \(Y\) is separable,
Theorem~\ref{thm:separable-grp-implies-rnp} shows that \(Y\) has the
Radon--Nikod\'ym property.
Thus every closed separable subspace of \(X\) has the Radon--Nikod\'ym
property. Since the Radon--Nikod\'ym property is separably determined
\cite[Theorem~III.3.2]{diestelUhl1977vector}
(see also \cite[Corollary~2.12]{pisier2016martingales}), \(X\)
has the Radon--Nikod\'ym property.
\end{proof}

\begin{remark}\label{rem:bounded-necessary-sequences}
The graph sequences to which the graphon representation property is applied
in the proofs of Proposition~\ref{prop:grp-implies-wsc} and
Theorem~\ref{thm:separable-grp-implies-rnp} are uniformly
\(L^\infty\)-bounded. Consequently, the same necessary conditions remain
valid for the bounded variant introduced in
Section~\ref{sec:bounded-representation}; see
Proposition~\ref{prop:bounded-inherited}(1).
\end{remark}

\section{Characterizations in three classes}
\label{sec:positive-characterizations}

This section characterizes the graphon representation property in three
natural classes, using the necessary conditions established in the previous
section. We first consider spaces with separable dual. We then establish a
separable-predual sufficient condition and apply it to Banach lattices and
dual Banach spaces.

\subsection{Spaces with separable dual}
\label{subsec:separable-dual}

\begin{theorem}[Separable-dual characterization]
\label{thm:separable-dual-characterization}
Let \(X\) be a Banach space with \(X^*\) separable. Then \(X\) has the
graphon representation property if and only if \(X\) is reflexive.
\end{theorem}

\begin{proof}
Assume first that \(X\) has the graphon representation property. By the
necessary condition proved in the previous section, \(X\) is weakly
sequentially complete. Since a Banach space with separable dual is reflexive
if and only if it is weakly sequentially complete
\cite[Section~3.5.1, p.~81]{pietsch2007history},
\(X\) is reflexive.

Conversely, assume that \(X\) is reflexive and that \(X^*\) is separable. Let
\((\mathbf G_n)\) be a sequence of
\(X\)-decorated target graphs which is uniformly \(L^p\)-bounded and such
that, for every \(X^*\)-decorated test graph \(\mathbf F\), the density
sequence \(t(\mathbf F,\mathbf G_n)\) converges. Let \(W_n\) be the equal-step
\(X\)-valued graphon associated with \(\mathbf G_n\). For every
\(X^*\)-decorated test graph \(\mathbf F\), we have
\(t(\mathbf F,\mathbf G_n)=t(\mathbf F,W_n)\).

Set \(B:=X^*\), so that \(B^*=X^{**}\). Since \(X\) is reflexive, the
canonical embedding \(J:X\to X^{**}\) is onto. Hence, after replacing \(W_n\)
by \(J\circ W_n\), we may regard \(W_n\) as a \(B^*\)-valued weak-* graphon.
This replacement does not change any density, since for all \(s,t\),
\(\langle x^*,J\circ W_n(s,t)\rangle=\langle x^*,W_n(s,t)\rangle\). It also
does not change the \(L^p\)-bounds, since \(J\) is an isometry. Therefore the
sequence \((W_n)\), viewed as a sequence of \(B^*\)-valued weak-* graphons,
satisfies the two hypotheses of Theorem~\ref{thm:kls-3.7}. Hence there exists
a \(B^*\)-valued weak-* graphon \(U\) such that \(t(\mathbf F,W_n)\to
t(\mathbf F,U)\) for every \(B\)-decorated test graph \(\mathbf F\).

Now we use reflexivity again. Since \(B^*=X^{**}=J(X)\), we may write
\(U=J\circ W\) for a unique symmetric \(X\)-valued function \(W\). The weak-*
measurability of \(U\) means that, for every \(x^*\in X^*\), the scalar
function \((s,t)\mapsto \langle x^*,U(s,t)\rangle\) is measurable. Since
\(U=J\circ W\), this is exactly
\((s,t)\mapsto \langle x^*,W(s,t)\rangle\). Thus \(W\) is weakly measurable
as an \(X\)-valued function. Because \(X^*\) is separable, the space \(X\) is
separable. Therefore, by the Pettis measurability theorem, the weakly
measurable function \(W\) is Bochner measurable.

Since \(J\) is an isometry and \(\|U(\cdot,\cdot)\|\in L^p\) for every
\(1\le p<\infty\), we have \(\|W(\cdot,\cdot)\|\in L^p\) for every
\(1\le p<\infty\). Hence \(W\) is an \(X\)-valued Bochner graphon. Finally, for
every \(X^*\)-decorated test graph \(\mathbf F\), the densities of
\(\mathbf F\) in \(U\) and in \(W\) agree. Therefore
\(t(\mathbf F,W_n)\to t(\mathbf F,W)\) for every \(X^*\)-decorated test graph
\(\mathbf F\). Thus the limiting densities are represented by an
\(X\)-valued Bochner graphon. Hence \(X\) has the graphon representation property.
This proves the equivalence.
\end{proof}

\subsection{A separable-predual sufficient condition}
\label{subsec:separable-predual}

We next turn to a sufficient condition for the graphon representation property.
The main idea is best illustrated by the case \(X=\ell_1\). Let
\((\mathbf G_n)\) be a sequence of \(\ell_1\)-decorated target graphs
that is uniformly \(L^p\)-bounded for every \(1\le p<\infty\), and suppose
that its densities converge against every
\(\ell_\infty\,(\cong\ell_1^*)\)-decorated test graph. Since
\(\ell_1\cong c_0^*\), applying
Theorem~\ref{thm:kls-3.7} with predual \(c_0\) produces an
\(\ell_1\)-valued (weak-*) graphon representing the limiting densities of all
\(c_0\)-decorated test graphs. The issue is whether the same graphon also
represents the limiting densities of all \(\ell_\infty\)-decorated test
graphs. Weak sequential completeness of \(\ell_1\) provides precisely this
upgrade. Since the same upgrade will also be used in the next section, we
isolate its general form in the following lemma.

\begin{lemma}
\label{lem:full-dual-upgrade}
Let \(X\) be a weakly sequentially complete Banach space, and let
\(D\subset X^*\) be weak-* dense. Let \((\mathbf G_n)\) be a
sequence of \(X\)-decorated target graphs such that
\(t(\mathbf F,\mathbf G_n)\) converges for every
\(X^*\)-decorated test graph \(\mathbf F\). Suppose that
\(W\in\mathcal W_{\mathrm{Boch}}(X)\) satisfies
\[
t(\mathbf F,\mathbf G_n)\longrightarrow t(\mathbf F,W)
\]
for every \(D\)-decorated test graph \(\mathbf F\). Then the same convergence
holds for every \(X^*\)-decorated test graph.
\end{lemma}

\begin{proof}
Let \(F\) be a loopless simple graph. If \(E(F)=\varnothing\), the assertion
is immediate. We may therefore assume that \(m:=|E(F)|\ge1\) and enumerate
the edges as \(E(F)=\{e_1,\dots,e_m\}\). For
\(\Phi_1,\dots,\Phi_m\in X^*\), let
\(\mathbf F^{\Phi_1,\dots,\Phi_m}\) denote the \(X^*\)-decorated test graph
whose underlying graph is \(F\) and whose edge \(e_i\) is decorated by
\(\Phi_i\).

Define \(\Lambda_F:(X^*)^m\to\mathbb{R}\) by
\begin{equation*}
\Lambda_F(\Phi_1,\dots,\Phi_m)
=
\lim_{n\to\infty}
\left[
 t(\mathbf F^{\Phi_1,\dots,\Phi_m},\mathbf G_n)
 -
 t(\mathbf F^{\Phi_1,\dots,\Phi_m},W)
\right].
\end{equation*}
This is well-defined, because the first density sequence converges by
assumption and the second term is fixed. Moreover, by the assumed convergence
for \(D\)-decorated test graphs,
\[
\Lambda_F(d_1,\dots,d_m)=0
\qquad
(d_1,\dots,d_m\in D).
\]

We claim that \(\Lambda_F\) is separately weak-* continuous on \((X^*)^m\).
That is, for every \(r=1,\dots,m\) and every fixed choice of
\(\Phi_1,\dots,\Phi_{r-1},\Phi_{r+1},\dots,\Phi_m\in X^*\), the map
\begin{equation*}
\Psi\mapsto
\Lambda_F(\Phi_1,\dots,\Phi_{r-1},\Psi,\Phi_{r+1},\dots,\Phi_m)
\end{equation*}
is continuous with respect to the weak-* topology \(\sigma(X^*,X)\). Without
loss of generality, we prove it for the first variable.

Fix \(\Phi_2,\dots,\Phi_m\in X^*\). Write
\(\mathbf G_n=(G_n,g_n)\), and write \(e_i=a_i b_i\) for
\(i=1,\dots,m\). Set
\[
\begin{aligned}
a_n
&:=
\frac{1}{|V(G_n)|^{|V(F)|}}
\sum_{\varphi:V(F)\to V(G_n)}
g_n\bigl(\varphi(a_1b_1)\bigr)
\prod_{q=2}^m
\left\langle
\Phi_q,g_n\bigl(\varphi(a_qb_q)\bigr)
\right\rangle\\
&\quad-
\int_{[0,1]^{V(F)}}
W(x_{a_1},x_{b_1})
\prod_{q=2}^m
\left\langle \Phi_q,W(x_{a_q},x_{b_q})\right\rangle
\prod_{i\in V(F)}dx_i
\in X.
\end{aligned}
\]

The first term is a finite sum. The integrand in the second term is Bochner
measurable, and its norm is bounded by
\[
\left(\prod_{q=2}^m\|\Phi_q\|\right)
\prod_{q=1}^m\|W(x_{a_q},x_{b_q})\|.
\]
By Hölder's inequality,
\[
\begin{aligned}
&
\int_{[0,1]^{V(F)}}
\prod_{q=1}^m\|W(x_{a_q},x_{b_q})\|
\prod_{i\in V(F)}dx_i
\\
&\quad\le
\prod_{q=1}^m
\left(
\int_{[0,1]^{V(F)}}
\|W(x_{a_q},x_{b_q})\|^m
\prod_{i\in V(F)}dx_i
\right)^{1/m}
\\
&\quad=
\|W\|_m^m<\infty.
\end{aligned}
\]
Thus \(a_n\) is well defined. For every \(\Psi\in X^*\), linearity and the
definition of homomorphism density give
\begin{equation*}
\langle \Psi,a_n\rangle
=
t(\mathbf F^{\Psi,\Phi_2,\dots,\Phi_m},\mathbf G_n)
-
t(\mathbf F^{\Psi,\Phi_2,\dots,\Phi_m},W).
\end{equation*}
The right-hand side converges as \(n\to\infty\), since
\(t(\mathbf F,\mathbf G_n)\) converges for every \(X^*\)-decorated test graph
\(\mathbf F\). Hence
\((a_n)\) is weakly Cauchy in \(X\).

Since \(X\) is weakly sequentially complete, there exists \(a\in X\) such that
\(a_n\to a\) weakly in \(X\). Therefore, for every \(\Psi\in X^*\),
\begin{align*}
\Lambda_F(\Psi,\Phi_2,\dots,\Phi_m)
&=
\lim_{n\to\infty}
\left[
 t(\mathbf F^{\Psi,\Phi_2,\dots,\Phi_m},\mathbf G_n)
 -
 t(\mathbf F^{\Psi,\Phi_2,\dots,\Phi_m},W)
\right] \\
&=
\lim_{n\to\infty}\langle \Psi,a_n\rangle \\
&=
\langle \Psi,a\rangle .
\end{align*}
Thus the map \(\Psi\mapsto\Lambda_F(\Psi,\Phi_2,\dots,\Phi_m)\) is evaluation
at the point \(a\in X\). Hence it is continuous for the weak-* topology
\(\sigma(X^*,X)\). The same argument applies to every variable, so
\(\Lambda_F\) is separately weak-* continuous.

Since \(D\) is weak-* dense in \(X^*\) and \(\Lambda_F=0\) on \(D^m\),
separate weak-* continuity gives, first,
\(\Lambda_F=0\) on \(X^*\times D^{m-1}\). Applying the same argument to the
second variable gives \(\Lambda_F=0\) on \((X^*)^2\times D^{m-2}\).
Continuing variable by variable, we obtain
\(\Lambda_F=0\) on \((X^*)^m\). Therefore, for all
\(\Phi_1,\dots,\Phi_m\in X^*\),
\begin{equation*}
t(\mathbf F^{\Phi_1,\dots,\Phi_m},\mathbf G_n)
\to
t(\mathbf F^{\Phi_1,\dots,\Phi_m},W).
\end{equation*}

Since \(F\) was arbitrary, this holds for every \(X^*\)-decorated test graph.
\end{proof}

We now apply the lemma to prove the following sufficient condition.

\begin{theorem}
\label{thm:separable-predual-sufficient}
Let \(X\) be a separable, weakly sequentially complete Banach space. Suppose
that there exists a Banach space \(B\) such that \(X\cong B^*\).
Then \(X\) has the graphon representation property.
\end{theorem}

\begin{proof}
The graphon representation property is invariant under isometric isomorphisms,
so it suffices to prove that \(B^*\) has this property. Since
\(X\cong B^*\), the space \(B^*\) is separable and weakly sequentially
complete. Since \(B^*\) is separable, so is \(B\).
Let \((\mathbf G_n)\) be a sequence of \(B^*\)-decorated target graphs which is
uniformly \(L^p\)-bounded and such that, for every \(B^{**}\)-decorated test
graph \(\mathbf F\), the density sequence
\(t(\mathbf F,\mathbf G_n)\) converges. Let \(W_n\) be the equal-step
\(B^*\)-valued graphon associated with \(\mathbf G_n\).

We regard \(B\) as a subspace of \(B^{**}\) by the canonical embedding. Thus
every \(B\)-decorated test graph is also a \(B^{**}\)-decorated test graph.
Hence the assumed convergence for all \(B^{**}\)-decorated test graphs
implies convergence for all \(B\)-decorated test graphs.

We now apply Theorem~\ref{thm:kls-3.7} with predual \(B\). The sequence
\((W_n)\) is uniformly \(L^p\)-bounded as a sequence of \(B^*\)-valued weak-*
graphons, and the densities against all \(B\)-decorated test graphs converge.
Therefore there exists a \(B^*\)-valued weak-* graphon
\(U:[0,1]^2\to B^*\) such that
\begin{equation}
\label{eq:kls-B-density}
t(\mathbf F,W_n)\to t(\mathbf F,U)
\end{equation}
for every \(B\)-decorated test graph \(\mathbf F\).

We next show that \(U\) is weakly measurable as a \(B^*\)-valued function.
By definition of weak-* measurability, for every \(b\in B\), the scalar function
\((s,t)\mapsto \langle b,U(s,t)\rangle\) is measurable.
Let \(\Phi\in B^{**}\). Since \(B^*\) is separable, the weak-* topology
\(\sigma(B^{**},B^*)\) is metrizable on bounded subsets of \(B^{**}\). By
Goldstine's theorem
\cite[Theorem~3.32]{aliprantisBurkinshaw2006positive}, \(B\) is weak-* dense
in \(B^{**}\). Hence there exists a sequence
\((b_j)\subset B\) such that \(b_j\to\Phi\) in
\(\sigma(B^{**},B^*)\). Therefore, for every \((s,t)\),
\(\langle b_j,U(s,t)\rangle\to\langle\Phi,U(s,t)\rangle\).
Each function
\((s,t)\mapsto \langle b_j,U(s,t)\rangle\) is measurable, so
\((s,t)\mapsto \langle \Phi,U(s,t)\rangle\) is measurable. Thus \(U\) is weakly
measurable as a \(B^*\)-valued function.

Since \(B^*\) is separable, the Pettis measurability theorem implies that
\(U\) is Bochner measurable. Together with the \(L^p\)-bound supplied by
Theorem~\ref{thm:kls-3.7}, \(U\) is a \(B^*\)-valued Bochner graphon. We
denote this graphon by \(W\).

It remains to show that \(W\) represents the limiting densities for all
\(B^{**}\)-decorated test graphs. For every \(B^{**}\)-decorated test graph
\(\mathbf F\), the definition of \(W_n\) gives
\(t(\mathbf F,W_n)=t(\mathbf F,\mathbf G_n)\).
It follows from this identity and \eqref{eq:kls-B-density} that
\(t(\mathbf F,\mathbf G_n)\to t(\mathbf F,W)\)
for every \(B\)-decorated test graph \(\mathbf F\). Applying
Lemma~\ref{lem:full-dual-upgrade} to \(B^*\) with \(D=B\), we obtain the same
convergence for every \(B^{**}\)-decorated test graph \(\mathbf F\). Thus
\(W\) represents all the limiting densities of \((\mathbf G_n)\), and
\(B^*\) has the graphon representation property. Hence \(X\) has the graphon
representation property.
\end{proof}

\subsection{Banach lattices and dual Banach spaces}
\label{subsec:banach-lattices}
\label{subsec:dual-banach-spaces}

We now use Theorem~\ref{thm:separable-predual-sufficient} to characterize the
graphon representation property for Banach lattices and dual Banach spaces. We
begin with Banach lattices.

\begin{theorem}[Banach lattice characterization]
\label{thm:banach-lattice-characterization}
Let \(X\) be a Banach lattice. Then \(X\) has the graphon representation
property if and only if \(X\) has the Radon--Nikodým property.
\end{theorem}

\begin{proof}
The necessity follows from Corollary~\ref{cor:grp-implies-rnp}.
Conversely, suppose that \(X\) has the Radon--Nikodým property. By
Theorem~\ref{thm:separable-descent}, it suffices to show that every closed
separable subspace \(E\subset X\) has the graphon representation property. Fix
such an \(E\). In a Banach lattice, every closed separable subspace is contained
in a closed separable sublattice, so there exists a closed separable sublattice
\(L\subset X\) containing \(E\)
\cite[p.~204, Exercise~9]{aliprantisBurkinshaw2006positive}. Since the
Radon--Nikodým property passes to closed subspaces, \(L\) has this property. By
Talagrand's theorem \cite[Theorem~1]{talagrand1983structure},
\(L\cong B^*\) for some Banach lattice \(B\). Moreover, \(L\) is weakly
sequentially complete. Indeed, \(c_0\) does not have the Radon--Nikodým
property, and therefore \(L\) cannot contain a closed subspace isomorphic to
\(c_0\) (see \cite[p.~60]{diestelUhl1977vector} or
\cite[p.~55]{pisier2016martingales}). It follows from
\cite[Theorem~4.60]{aliprantisBurkinshaw2006positive} that \(L\) is weakly
sequentially complete. Therefore
Theorem~\ref{thm:separable-predual-sufficient} shows that \(L\) has the graphon
representation property. Theorem~\ref{thm:separable-descent} then gives the
same property for \(E\). Since \(E\) was arbitrary, another application of
Theorem~\ref{thm:separable-descent} gives the graphon representation property
for \(X\).
\end{proof}

We next turn to dual Banach spaces. Recall that, as one of its
equivalent definitions, a Banach space is \emph{Asplund} if every closed separable
subspace has separable dual \cite[Theorem~2.34]{phelps1993convex}. The
Namioka--Phelps--Stegall theorem states that a
Banach space \(B\) is Asplund if and only if \(B^*\) has the Radon--Nikodým
property \cite[Theorem~5.7]{phelps1993convex}.

\begin{theorem}[Dual Banach space characterization]
\label{thm:dual-banach-characterization}
Let \(X\) be a dual Banach space. Then \(X\) has the graphon representation
property if and only if \(X\) is weakly sequentially complete and has the
Radon--Nikodým property.
\end{theorem}

\begin{proof}
The necessity follows from Proposition~\ref{prop:grp-implies-wsc} and
Corollary~\ref{cor:grp-implies-rnp}.

Conversely, suppose that \(X\) is weakly sequentially complete and has the
Radon--Nikodým property. Since \(X\) is a dual Banach space, there exists a
Banach space \(B\) such that \(X\cong B^*\). Since these properties are
invariant under isometric isomorphisms, \(B^*\) is weakly sequentially complete
and has the Radon--Nikodým property. It is therefore enough to prove that
\(B^*\) has the graphon representation property. By the
Namioka--Phelps--Stegall theorem, \(B\) is Asplund.
Suppose first that \(B\) is separable. Since \(B\) is Asplund, \(B^*\) is
separable. Thus Theorem~\ref{thm:separable-predual-sufficient} shows that
\(B^*\) has the graphon representation property.

We next consider the case where \(B\) is nonseparable. By
\cite[Theorem~6]{cuthFabian2015projections}, after restricting the family there
to its members of density \(\aleph_0\), there exist a family \(\mathcal V\) of
closed separable subspaces of \(B\) and a family
\((Y_V)_{V\in\mathcal V}\) of closed subspaces of \(B^*\) with the following
properties.
\begin{enumerate}
\item[(1)] \(B^*=\bigcup_{V\in\mathcal V}Y_V\).
\item[(2)] For any \(V_1,V_2\in\mathcal V\), there exists \(V\in\mathcal V\)
such that \(V_1\cup V_2\subset V\).
\item[(3)] If \(V_1,V_2\in\mathcal V\) and \(V_1\subset V_2\), then
\(Y_{V_1}\subset Y_{V_2}\).
\item[(4)] If \((V_n)\) is an increasing sequence in \(\mathcal V\), then
\(V:=\overline{\bigcup_{n=1}^{\infty}V_n}\in\mathcal V\) and
\(Y_V=\overline{\bigcup_{n=1}^{\infty}Y_{V_n}}\).
\item[(5)] For every \(V\in\mathcal V\), the restriction map
\(Y_V\to V^*\) is a surjective isometry.
\end{enumerate}

Let \(Y\subset B^*\) be a closed separable subspace, and choose a dense
sequence \((y_n)\) in \(Y\). For each \(n\), choose \(U_n\in\mathcal V\) such
that \(y_n\in Y_{U_n}\), using (1). By (2), we can construct an increasing
sequence \((V_n)\) in \(\mathcal V\) such that
\(U_n\subset V_n\) for every \(n\). It follows from (3) that
\(y_n\in Y_{V_n}\). Set
\(V=\overline{\bigcup_{n=1}^{\infty}V_n}\). By (4), we have
\(Y_V=\overline{\bigcup_{n=1}^{\infty}Y_{V_n}}\), and therefore
\(Y\subset Y_V\). By (5), the restriction map identifies \(Y_V\)
isometrically with \(V^*\). Since \(V\) is separable and \(B\) is Asplund,
\(V^*\) is separable, and hence so is \(Y_V\). Since \(Y_V\) is a closed subspace of
the weakly sequentially complete space \(B^*\), it is weakly sequentially
complete. Theorem~\ref{thm:separable-predual-sufficient} therefore shows that
\(Y_V\) has the graphon representation property, and
Theorem~\ref{thm:separable-descent} gives the same property for \(Y\). Since
\(Y\) was arbitrary, another application of Theorem~\ref{thm:separable-descent}
shows that \(B^*\) has the graphon representation property.
\end{proof}

\begin{remark}
A characterization of the graphon representation property for
general Banach spaces remains open. In particular, it is unknown whether there
exists a weakly sequentially complete Banach space with the Radon--Nikodým
property that does not have the graphon representation property.
\end{remark}

\section{Bounded graphon representability}
\label{sec:bounded-representation}

In this section, we first introduce the bounded graphon representation property
and show that, for every Banach space, this property is equivalent to weak
sequential completeness and the Radon--Nikodým property. The main
graph-theoretic ingredient in the proof is the use of two-labelled graphs
whose labelled vertices are nonadjacent.

\begin{definition}
\label{def:bgrp}
Let \(X\) be a Banach space. We say that \(X\) has the \emph{bounded graphon
representation property} if the following holds: if a sequence
\((\mathbf G_n)\) of \(X\)-decorated target graphs satisfies
\begin{enumerate}
\item \(\sup_n\|\mathbf G_n\|_\infty<\infty\);
\item for every \(X^*\)-decorated test graph \(\mathbf F\), the sequence
\((t(\mathbf F,\mathbf G_n))\) converges,
\end{enumerate}
then there exists a bounded \(X\)-valued Bochner graphon
\(W\in\mathcal W_{\mathrm{Boch}}(X)\) such that
\(t(\mathbf F,\mathbf G_n)\to t(\mathbf F,W)\) for every
\(X^*\)-decorated test graph \(\mathbf F\).
\end{definition}

As observed in Remark~\ref{rem:bounded-necessary-sequences}, the arguments
establishing weak sequential completeness and, in the separable case, the
Radon--Nikodým property use only uniformly \(L^\infty\)-bounded graph
sequences. Moreover, the proof of Theorem~\ref{thm:separable-descent} applies
without change to the bounded graphon representation property. We therefore
obtain the following.

\begin{proposition}
\label{prop:bounded-inherited}
The following statements hold.
\begin{enumerate}
\item If \(X\) has the bounded graphon representation property, then \(X\) is
weakly sequentially complete and has the Radon--Nikodým property.

\item A Banach space \(X\) has the bounded graphon representation property if
and only if every closed separable subspace of \(X\) has it.
\end{enumerate}
\end{proposition}

The following is the main result of this section. We state it now and then
establish the preparatory results needed for the sufficiency.

\begin{theorem}[Bounded characterization]
\label{thm:bounded-characterization}
Let \(X\) be a Banach space. Then \(X\) has the bounded graphon representation
property if and only if \(X\) is weakly sequentially complete and has the
Radon--Nikodým property.
\end{theorem}

\subsection{Graphs with two labelled vertices}

Let \(Y\) be a Banach space.
A two-labelled \(Y\)-decorated test graph is a triple
\((\mathbf F;\rho_1,\rho_2)\), where \(\mathbf F=(F,f)\) is a
\(Y\)-decorated test graph and \(\rho_1,\rho_2\in V(F)\) are distinct
vertices. We regard \(\rho_1\) as the vertex with label \(1\), and
\(\rho_2\) as the vertex with label \(2\). When the labelled vertices are
understood, we simply write \(\mathbf F\).

We use the standard gluing product of labelled graphs: the product is
obtained by taking the disjoint union and identifying vertices with the
same label; see \cite[Section~4.2]{lovasz2012largeNetworks}. Since the two labels
are assigned to distinct vertices, this operation produces a simple graph.

For decorated graphs, the decoration is inherited from the original
decorated edges whenever this inheritance is unambiguous. In the
two-labelled case, the only possible ambiguity comes from the edge between
the two labelled vertices. Indeed, if both factors contain this labelled
edge, then after gluing there is only one such edge, but there are two
possible decorations for it. We avoid this ambiguity by using the following
class of labelled graphs.

We say that a two-labelled \(Y\)-decorated test graph
\((\mathbf F;\rho_1,\rho_2)\) is \emph{nonadjacent} if
\(\rho_1\rho_2\notin E(F)\).
If \((\mathbf F;\rho_1,\rho_2)\) and \((\mathbf G;\eta_1,\eta_2)\) are
nonadjacent two-labelled \(Y\)-decorated test graphs, then their gluing product
has a canonically inherited decoration and is again nonadjacent. We denote this
product by \((\mathbf F;\rho_1,\rho_2)\vee(\mathbf G;\eta_1,\eta_2)\), or
simply by \(\mathbf F\vee\mathbf G\) when the labels are understood; see
Figure~\ref{fig:rooted-gluing}.

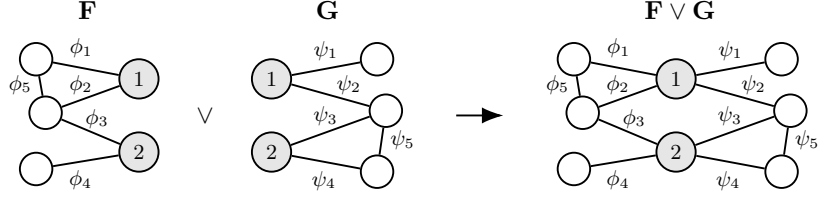
\begin{figure}
\centering
\resizebox{0.85\textwidth}{!}{%
\begin{tikzpicture}[
  x=1cm,y=1cm,
  ordinary/.style={
    circle,
    draw,
    thick,
    fill=white,
    inner sep=1pt,
    minimum size=4.8mm
  },
  root/.style={
    circle,
    draw,
    thick,
    fill=black!10,
    inner sep=1pt,
    minimum size=5.8mm
  },
  Aedge/.style={thick},
  Cedge/.style={thick},
  arrow/.style={-{Latex[length=3.2mm]}, thick},
  lab/.style={font=\small, inner sep=0.5pt},
  title/.style={font=\large},
  every node/.style={font=\small}
]

\begin{scope}[shift={(0,0)}]
\node[title] at (1.15,1.56) {$\mathbf F$};

  \node[ordinary] (aT) at (0.39,0.83) {};
  \node[ordinary] (aM) at (0.53,0.04) {};
  \node[ordinary] (aB) at (0.39,-0.79) {};

  \node[root] (arT) at (1.91,0.54) {$1$};
  \node[root] (arB) at (1.91,-0.54) {$2$};

  \draw[Aedge] (aT) -- node[lab,pos=0.43,above=1.2mm] {$\phi_1$} (arT);
  \draw[Aedge] (aM) -- node[lab,pos=0.33,above=0.6mm] {$\phi_2$} (arT);
  \draw[Aedge] (aM) -- node[lab,pos=0.61,above=0.6mm] {$\phi_3$} (arB);
  \draw[Aedge] (aB) -- node[lab,pos=0.42,below=1.0mm] {$\phi_4$} (arB);
  \draw[Aedge] (aT) -- node[lab,pos=0.43,left=1.0mm] {$\phi_5$} (aM);
\end{scope}

\node[title] at (2.89,0) {$\vee$};

\begin{scope}[shift={(3.45,0)}]
\node[title] at (1.25,1.56) {$\mathbf G$};

  \node[root] (crT) at (0.42,0.54) {$1$};
  \node[root] (crB) at (0.42,-0.54) {$2$};

  \node[ordinary] (cT) at (1.98,0.83) {};
  \node[ordinary] (cM) at (2.11,0.07) {};
  \node[ordinary] (cB) at (1.98,-0.83) {};

  \draw[Cedge] (crT) -- node[lab,pos=0.50,above=1.0mm] {$\psi_1$} (cT);
  \draw[Cedge] (crT) -- node[lab,pos=0.74,above=0.6mm] {$\psi_2$} (cM);
  \draw[Cedge] (crB) -- node[lab,pos=0.44,above=0.7mm] {$\psi_3$} (cM);
  \draw[Cedge] (crB) -- node[lab,pos=0.47,below=1.0mm] {$\psi_4$} (cB);
  \draw[Cedge] (cM) -- node[lab,pos=0.46,right=1.0mm] {$\psi_5$} (cB);
\end{scope}

\draw[arrow] (6.60,0) -- (7.30,0);

\begin{scope}[shift={(7.55,0)}]
\node[title] at (2.36,1.56) {$\mathbf F\vee\mathbf G$};

  \node[ordinary] (gAT) at (0.79,0.83) {};
  \node[ordinary] (gAM) at (0.92,0.04) {};
  \node[ordinary] (gAB) at (0.79,-0.79) {};

  \node[root] (grT) at (2.30,0.54) {$1$};
  \node[root] (grB) at (2.30,-0.54) {$2$};

  \node[ordinary] (gCT) at (3.86,0.83) {};
  \node[ordinary] (gCM) at (3.99,0.07) {};
  \node[ordinary] (gCB) at (3.86,-0.83) {};

  \draw[Aedge] (gAT) -- node[lab,pos=0.43,above=1.2mm] {$\phi_1$} (grT);
  \draw[Aedge] (gAM) -- node[lab,pos=0.33,above=0.6mm] {$\phi_2$} (grT);
  \draw[Aedge] (gAM) -- node[lab,pos=0.61,above=0.6mm] {$\phi_3$} (grB);
  \draw[Aedge] (gAB) -- node[lab,pos=0.42,below=1.0mm] {$\phi_4$} (grB);
  \draw[Aedge] (gAT) -- node[lab,pos=0.43,left=1.0mm] {$\phi_5$} (gAM);

  \draw[Cedge] (grT) -- node[lab,pos=0.50,above=1.0mm] {$\psi_1$} (gCT);
  \draw[Cedge] (grT) -- node[lab,pos=0.74,above=0.6mm] {$\psi_2$} (gCM);
  \draw[Cedge] (grB) -- node[lab,pos=0.44,above=0.7mm] {$\psi_3$} (gCM);
  \draw[Cedge] (grB) -- node[lab,pos=0.47,below=1.0mm] {$\psi_4$} (gCB);
  \draw[Cedge] (gCM) -- node[lab,pos=0.46,right=1.0mm] {$\psi_5$} (gCB);
\end{scope}

\end{tikzpicture}%
}
\caption{Construction of \(\mathbf F\vee\mathbf G\). The labelled vertices
\(1\) and \(2\) are identified. Edges with decorations \(\phi_i\in Y_0\) are
inherited from \(\mathbf F\), and edges with decorations \(\psi_j\in Y_0\) are
inherited from \(\mathbf G\).}
\label{fig:rooted-gluing}
\end{figure}

Assume now that \(Y\) is separable, and let
\(U:[0,1]^2\to Y^*\) be a bounded weak-* graphon. For a two-labelled
\(Y\)-decorated test graph
\((\mathbf F;\rho_1,\rho_2)\), define
\(\tau_{\mathbf F}^U\in L^\infty([0,1]^2)\) by
\[
\tau_{\mathbf F}^U(x_{\rho_1},x_{\rho_2})
:=
\int_{[0,1]^{V(F)\setminus\{\rho_1,\rho_2\}}}
\prod_{ij\in E(F)}
\langle f_{ij},U(x_i,x_j)\rangle
\prod_{i\in V(F)\setminus\{\rho_1,\rho_2\}} dx_i.
\]
Since \(U\) is bounded and weak-* measurable, this function is well-defined
a.e., bounded, and measurable. If \(\mathbf F\) consists only of the two labelled
vertices and has no edge, then we use the convention \(\tau_{\mathbf F}^U=1\).

Let \(Y_0\subset Y\). We write \(\mathscr R_U(Y_0)\) for the linear span of
the functions \(\tau_{\mathbf F}^U\), where
\((\mathbf F;\rho_1,\rho_2)\) ranges over all nonadjacent two-labelled
\(Y_0\)-decorated test graphs.

\begin{lemma}
Let \(Y\) be a separable Banach space, let \(Y_0\subset Y\) be countable, and let
\(U:[0,1]^2\to Y^*\) be a bounded weak-* graphon. Then
\(\mathscr R_U(Y_0)\) is a countably generated unital subalgebra of
\(L^\infty([0,1]^2)\).
\end{lemma}

\begin{proof}
We first prove closure under products of the generators
\((\tau_{\mathbf F}^U)_{\mathbf F}\). Let
\(\mathbf F=(F,f)\) and \(\mathbf G=(G,g)\) be nonadjacent two-labelled
\(Y_0\)-decorated test graphs. Write \(V(F)=\{1,\ldots,p\}\) and
\(V(G)=\{1,\ldots,q\}\), and assume without loss of generality that the
labelled vertices are \(1,2\) in both graphs. Let
\(\mathbf F\vee\mathbf G\) be the gluing product defined above. By the
preceding discussion, \(\mathbf F\vee\mathbf G\) is again a nonadjacent
two-labelled \(Y_0\)-decorated test graph.

Use \(x_1,x_2\) for the common labelled variables, use
\(y_3,\ldots,y_p\) for the non-labelled vertices of \(\mathbf F\), and use
\(z_3,\ldots,z_q\) for the non-labelled vertices of \(\mathbf G\). Put
\[
P_F(x_1,x_2,y)
:=
\prod_{\substack{3\le i\le p\\ 1i\in E(F)}}
\langle f_{1i},U(x_1,y_i)\rangle
\prod_{\substack{3\le i\le p\\ 2i\in E(F)}}
\langle f_{2i},U(x_2,y_i)\rangle
\prod_{\substack{3\le i<j\le p\\ ij\in E(F)}}
\langle f_{ij},U(y_i,y_j)\rangle
\]
and similarly
\[
P_G(x_1,x_2,z)
:=
\prod_{\substack{3\le i\le q\\ 1i\in E(G)}}
\langle g_{1i},U(x_1,z_i)\rangle
\prod_{\substack{3\le i\le q\\ 2i\in E(G)}}
\langle g_{2i},U(x_2,z_i)\rangle
\prod_{\substack{3\le i<j\le q\\ ij\in E(G)}}
\langle g_{ij},U(z_i,z_j)\rangle .
\]
Empty products and integrals over \([0,1]^0\) are interpreted as \(1\).
Then, for a.e. \((x_1,x_2)\in[0,1]^2\),
\[
\tau_{\mathbf F}^U(x_1,x_2)
=\int_{[0,1]^{p-2}}P_F(x_1,x_2,y)\,dy
\]
and
\[
\tau_{\mathbf G}^U(x_1,x_2)
=\int_{[0,1]^{q-2}}P_G(x_1,x_2,z)\,dz .
\]
Since \(U\) is bounded and the graphs are finite, the two integrands are
bounded. Fubini's theorem gives, for a.e. \((x_1,x_2)\in[0,1]^2\),
\[
\begin{aligned}
\tau_{\mathbf F}^U(x_1,x_2)\tau_{\mathbf G}^U(x_1,x_2)
&=
\int_{[0,1]^{p-2}}\int_{[0,1]^{q-2}}
P_F(x_1,x_2,y)P_G(x_1,x_2,z)\,dz\,dy .
\end{aligned}
\]
By the definition of \(\mathbf F\vee\mathbf G\), the integrand for
\(\tau_{\mathbf F\vee\mathbf G}^U\) is exactly
\[
P_F(x_1,x_2,y)P_G(x_1,x_2,z).
\]
Therefore
\(\tau_{\mathbf F}^U\tau_{\mathbf G}^U
=\tau_{\mathbf F\vee\mathbf G}^U\) a.e. on \([0,1]^2\).
Thus products of generators are again generators.

If \(r=\sum_i a_i\tau_{\mathbf F_i}^U\) and
\(s=\sum_j b_j\tau_{\mathbf G_j}^U\), then
\(rs=\sum_{i,j}a_i b_j\,\tau_{\mathbf F_i\vee\mathbf G_j}^U
\in \mathscr R_U(Y_0)\).
Hence \(\mathscr R_U(Y_0)\) is closed under multiplication. It is closed under
addition and scalar multiplication by definition. For the graph \(\mathbf F\)
with only the two labelled vertices and no edge,
\(\tau_{\mathbf F}^U=1\) by convention. Hence \(\mathscr R_U(Y_0)\) is
unital.

It remains to prove countable generation. Since \(Y_0\) is countable, for each
fixed number \(k\ge2\) of vertices there are only countably many
\(Y_0\)-decorated graphs. Hence there are only countably many nonadjacent
two-labelled \(Y_0\)-decorated test graphs with \(k\) vertices, and countably
many in total after taking the union over \(k\ge2\). Therefore the generators
\((\tau_{\mathbf F}^U)_{\mathbf F}\) form a countable family. This proves
that \(\mathscr R_U(Y_0)\) is a countably generated unital subalgebra of
\(L^\infty([0,1]^2)\).
\end{proof}

\begin{lemma}
Let \(Y\) be a separable Banach space, let \(Y_0\subset Y\), and let
\(U:[0,1]^2\to Y^*\) be a bounded weak-* graphon.

Let \((\mathbf H;\rho_1,\rho_2)\) be a nonadjacent two-labelled
\(Y_0\)-decorated test graph, and write \(\mathbf H=(H,h)\). Enumerate the
edges as \(E(H)=\{e_1,\ldots,e_N\}\). Write \(e_j=a_jb_j\) and
\(h_j:=h_{e_j}\).
Let \(0\le r\le N\). For each \(i=1,\ldots,r\), let
\((\mathbf K_i;\alpha_i,\beta_i)\) be a nonadjacent two-labelled
\(Y_0\)-decorated test graph.

Define, for a.e.
\((x_{\rho_1},x_{\rho_2})\in[0,1]^2\),
\[
\begin{aligned}
C_r(x_{\rho_1},x_{\rho_2})
:=
\int_{[0,1]^{V(H)\setminus\{\rho_1,\rho_2\}}}
&
\left(
\prod_{i=1}^{r}
\tau_{\mathbf K_i}^U(x_{a_i},x_{b_i})
\right) \\
&\times
\left(
\prod_{j=r+1}^{N}
\langle h_j,U(x_{a_j},x_{b_j})\rangle
\right)
\prod_{i\in V(H)\setminus\{\rho_1,\rho_2\}} dx_i.
\end{aligned}
\]
Then \(C_r\in \mathscr R_U(Y_0)\).\footnote{Here the first product in the
integrand is \(1\) if \(r=0\), and the second product is \(1\) if \(r=N\).}
\end{lemma}

\begin{proof}
Write \(\mathbf H=(H,h)\) and \(\mathbf K_i=(K_i,k^i)\). If \(r=0\), then
\(C_0=\tau_{\mathbf H}^U\), so the assertion is clear. We may therefore assume
that \(r\ge1\).

We define a two-labelled \(Y_0\)-decorated test graph
\(\mathbf H^{(r)}\) as follows. Starting from \(\mathbf H\), delete the
original edges \(e_1,\ldots,e_r\). For each \(i=1,\ldots,r\), glue
\(\mathbf K_i\) in the place of \(e_i\), identifying \(\alpha_i\) with
\(a_i\), and \(\beta_i\) with \(b_i\). The two labels of
\(\mathbf H^{(r)}\) are inherited from \(\mathbf H\). Decorations are
inherited from the remaining edges of \(\mathbf H\) and from the edges of the
inserted graphs \(\mathbf K_i\).
By construction, \(\mathbf H^{(r)}\) is a nonadjacent two-labelled
\(Y_0\)-decorated test graph.
Figure~\ref{fig:finite-edge-replacement} illustrates the construction of
\(\mathbf H^{(r)}\).

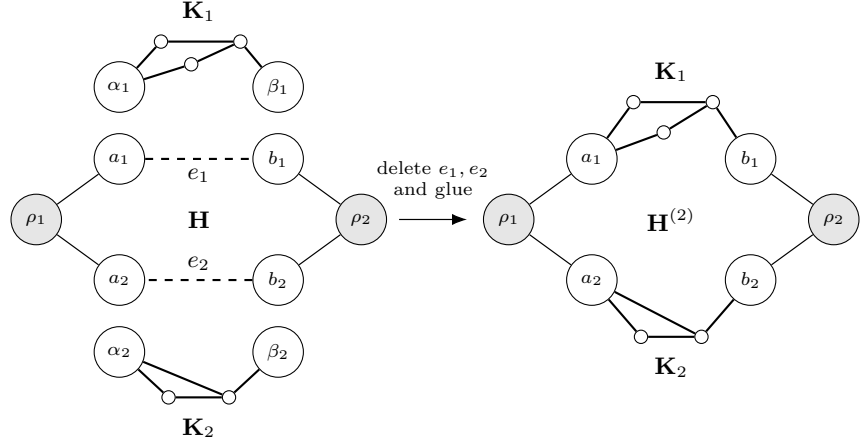
\begin{figure}[t]
\centering
\begin{tikzpicture}[
  x=1cm,
  y=1cm,
  every node/.style={font=\small},
  v/.style={circle,draw,fill=white,inner sep=1.2pt,minimum size=5pt},
  named/.style={circle,draw,fill=white,inner sep=0pt,minimum size=19pt,
    font=\scriptsize},
  root/.style={named,fill=black!10},
  edge/.style={line width=0.45pt},
  markedge/.style={line width=0.8pt,dashed},
  inserted/.style={line width=0.8pt},
  arrow/.style={-{Latex[length=2mm]},line width=0.45pt}
]

\begin{scope}
\node[root] (r1) at (0.15,0) {\(\rho_1\)};
\node[root] (r2) at (4.45,0) {\(\rho_2\)};
\node[named] (a1) at (1.25,0.8) {\(a_1\)};
\node[named] (b1) at (3.35,0.8) {\(b_1\)};
\node[named] (a2) at (1.25,-0.8) {\(a_2\)};
\node[named] (b2) at (3.35,-0.8) {\(b_2\)};

\draw[edge] (r1)--(a1);
\draw[markedge] (a1)--(b1) node[midway,below] {\(e_1\)};
\draw[edge] (b1)--(r2);
\draw[edge] (r2)--(b2);
\draw[markedge] (b2)--(a2) node[midway,above] {\(e_2\)};
\draw[edge] (a2)--(r1);
\node at (2.3,0) {\(\mathbf H\)};

\node[named] (al1) at (1.25,1.75) {\(\alpha_1\)};
\node[named] (be1) at (3.35,1.75) {\(\beta_1\)};
\node[v] (u11) at (1.8,2.35) {};
\node[v] (u12) at (2.85,2.35) {};
\node[v] (u13) at (2.2,2.05) {};
\draw[inserted] (al1)--(u11)--(u12)--(be1);
\draw[inserted] (al1)--(u13)--(u12);
\node at (2.3,2.75) {\(\mathbf K_1\)};

\node[named] (al2) at (1.25,-1.75) {\(\alpha_2\)};
\node[named] (be2) at (3.35,-1.75) {\(\beta_2\)};
\node[v] (u21) at (1.9,-2.35) {};
\node[v] (u22) at (2.7,-2.35) {};
\draw[inserted] (al2)--(u21)--(u22)--(be2);
\draw[inserted] (al2)--(u22);
\node at (2.3,-2.75) {\(\mathbf K_2\)};
\end{scope}

\draw[arrow] (4.95,0)--(5.85,0);
\node[align=center,font=\scriptsize] at (5.4,0.52)
  {delete \(e_1,e_2\)\\and glue};

\begin{scope}[shift={(6.25,0)}]
\node[root] (R1) at (0.15,0) {\(\rho_1\)};
\node[root] (R2) at (4.45,0) {\(\rho_2\)};
\node[named] (A1) at (1.25,0.8) {\(a_1\)};
\node[named] (B1) at (3.35,0.8) {\(b_1\)};
\node[named] (A2) at (1.25,-0.8) {\(a_2\)};
\node[named] (B2) at (3.35,-0.8) {\(b_2\)};

\draw[edge] (R1)--(A1);
\draw[edge] (B1)--(R2);
\draw[edge] (R2)--(B2);
\draw[edge] (A2)--(R1);
\node at (2.3,0) {\(\mathbf H^{(2)}\)};

\node[v] (U11) at (1.8,1.55) {};
\node[v] (U12) at (2.85,1.55) {};
\node[v] (U13) at (2.2,1.15) {};
\draw[inserted] (A1)--(U11)--(U12)--(B1);
\draw[inserted] (A1)--(U13)--(U12);
\node at (2.3,1.95) {\(\mathbf K_1\)};

\node[v] (U21) at (1.9,-1.55) {};
\node[v] (U22) at (2.7,-1.55) {};
\draw[inserted] (A2)--(U21)--(U22)--(B2);
\draw[inserted] (A2)--(U22);
\node at (2.3,-1.95) {\(\mathbf K_2\)};
\end{scope}

\end{tikzpicture}
\caption{Construction of \(\mathbf H^{(2)}\). The left panel shows
\(\mathbf H\) together with the graphs \(\mathbf K_1\) and \(\mathbf K_2\).
The marked edges \(e_i=a_i b_i\) are deleted, and
\(\mathbf K_i\) is glued in their place by identifying \(\alpha_i\) with
\(a_i\) and \(\beta_i\) with \(b_i\). The right panel shows the resulting
graph \(\mathbf H^{(2)}\); the thicker edges belong to the inserted graphs,
and all edge decorations are inherited from the original graphs.}
\label{fig:finite-edge-replacement}
\end{figure}

For each \(i=1,\ldots,r\), put
\(I_i:=V(K_i)\setminus\{\alpha_i,\beta_i\}\). We assign variables \(x_v\) to
the vertices \(v\in V(H)\), and variables \(w_u^i\) to the vertices
\(u\in I_i\). For each \(i\), since
\(\mathbf K_i\) is nonadjacent, there is no edge between \(\alpha_i\) and
\(\beta_i\). Thus the edges of \(K_i\) are partitioned into
\(E^i_{II}:=\{mn\in E(K_i):m,n\in I_i\}\),
\(E^i_{\alpha I}:=\{\alpha_i u\in E(K_i):u\in I_i\}\), and
\(E^i_{\beta I}:=\{\beta_i u\in E(K_i):u\in I_i\}\).
For \(x_{a_i},x_{b_i}\in[0,1]\) and
\(w^i=(w_u^i)_{u\in I_i}\in[0,1]^{I_i}\), define
\[
\begin{aligned}
A_i(x_{a_i},x_{b_i},w^i)
&:=
\left(\prod_{mn\in E^i_{II}}
\langle k^i_{mn},U(w_m^i,w_n^i)\rangle\right)
\left(\prod_{\alpha_i u\in E^i_{\alpha I}}
\langle k^i_{\alpha_i u},U(x_{a_i},w_u^i)\rangle\right)\\
&\quad\times
\left(\prod_{\beta_i u\in E^i_{\beta I}}
\langle k^i_{\beta_i u},U(x_{b_i},w_u^i)\rangle\right).
\end{aligned}
\]
Empty products are interpreted as \(1\).

By the definition of \(\tau_{\mathbf H^{(r)}}^U\), for a.e.
\((x_{\rho_1},x_{\rho_2})\in[0,1]^2\),
\[
\begin{aligned}
\tau_{\mathbf H^{(r)}}^U(x_{\rho_1},x_{\rho_2})
&=
\int_{[0,1]^{V(H)\setminus\{\rho_1,\rho_2\}}}
\int_{\prod_{i=1}^r[0,1]^{I_i}}
\left(
\prod_{j=r+1}^{N}
\langle h_j,U(x_{a_j},x_{b_j})\rangle
\right)\\
&\quad\times
\left(
\prod_{i=1}^{r}
A_i(x_{a_i},x_{b_i},w^i)
\right)
\,dw^1\cdots dw^r
\prod_{v\in V(H)\setminus\{\rho_1,\rho_2\}} dx_v.
\end{aligned}
\]
The graphon \(U\) is bounded and all graphs are finite, so all integrands are
bounded. Hence Fubini's theorem applies. For each \(i\), the definition of
\(\tau_{\mathbf K_i}^U\) gives
\[
\int_{[0,1]^{I_i}}A_i(x_{a_i},x_{b_i},w^i)\,dw^i
=\tau_{\mathbf K_i}^U(x_{a_i},x_{b_i})
\]
for a.e. \((x_{a_i},x_{b_i})\). Therefore we may integrate the internal
variables of the inserted graphs successively and obtain, for a.e.
\((x_{\rho_1},x_{\rho_2})\in[0,1]^2\),
\[
\begin{aligned}
\tau_{\mathbf H^{(r)}}^U(x_{\rho_1},x_{\rho_2})
&=
\int_{[0,1]^{V(H)\setminus\{\rho_1,\rho_2\}}}
\left(
\prod_{i=1}^{r}
\tau_{\mathbf K_i}^U(x_{a_i},x_{b_i})
\right)\\
&\quad\times
\left(
\prod_{j=r+1}^{N}
\langle h_j,U(x_{a_j},x_{b_j})\rangle
\right)
\prod_{i\in V(H)\setminus\{\rho_1,\rho_2\}} dx_i.
\end{aligned}
\]
The right-hand side is \(C_r(x_{\rho_1},x_{\rho_2})\) by definition. Hence
\(C_r=\tau_{\mathbf H^{(r)}}^U\) a.e. on \([0,1]^2\).
Since \(\mathbf H^{(r)}\) is a nonadjacent two-labelled
\(Y_0\)-decorated test graph, this gives \(C_r\in\mathscr R_U(Y_0)\).
This proves the lemma.
\end{proof}

\subsection{Closure and measurability}

We first establish the following functional-analytic lemma.

\begin{lemma}\label{lem:bounded-pointwise-closure}
Let \((\Omega,\mathcal F,\mu)\) be a complete probability space. Let
\(\mathscr R\subset L^\infty(\Omega)\) be a countably generated unital
subalgebra, and put \(\mathcal A:=\sigma(\mathscr R)\). Let
\(\mathcal D\subset L^\infty(\mathcal A)\) be a vector subspace which
contains \(\mathscr R\) and is closed under bounded pointwise a.e.
limits.\footnote{Whenever \(f_n\in\mathcal D\),
\(\sup_n\|f_n\|_\infty<\infty\), and \(f_n\to f\) pointwise a.e., with
\(f\in L^\infty(\mathcal A)\), then \(f\in\mathcal D\).}
Then \(\mathcal D=L^\infty(\mathcal A)\).
\end{lemma}

Note that, since the elements of \(L^\infty(\Omega)\) are equivalence classes,
\(\sigma(\mathscr R)\) denotes the smallest sub-\(\sigma\)-algebra of
\(\mathcal F\) with respect to which every representative of every element of
\(\mathscr R\) is measurable.

\begin{proof}
Choose a countable family \(r_1,r_2,\ldots\in\mathscr R\) which generates
\(\mathscr R\) as a unital algebra. Changing these representatives on null
sets if necessary, choose constants \(C_j>0\) such that
\(|r_j(\omega)|\le C_j\) for every \(\omega\in\Omega\). Note that \(r_j\)
denotes a representative of the corresponding equivalence class. Put
\(\mathcal A_0:=\sigma(r_j:j\ge1)\).
By the definition of \(\sigma(\mathscr R)\), the \(\sigma\)-algebra
\(\mathcal A\) is generated by \(\mathcal A_0\) and all \(\mu\)-null sets.

Define \(\Theta:\Omega\to\mathbb R^{\mathbb N}\) by
\(\Theta(\omega):=(r_j(\omega))_{j\ge1}\), and set
\(K:=\overline{\Theta(\Omega)}\subset
\prod_{j=1}^\infty[-C_j,C_j]\).
Then \(K\) is compact and metrizable. Since the Borel \(\sigma\)-algebra of
\(K\) is generated by finite-coordinate cylinders, we have
\(\mathcal A_0=\sigma(r_j:j\ge1)=\Theta^{-1}(\mathcal B(K))\).

Let \(\pi_j:K\to\mathbb R\) be the \(j\)-th coordinate projection, and let
\(\mathcal P\) be the unital subalgebra of \(C(K)\) generated by the coordinate
functions \(\pi_1,\pi_2,\ldots\).
If \(P\in\mathcal P\), then \(P\circ\Theta\) is a polynomial in finitely many
of the functions \(r_j\) and the constant \(1\). Since \(\mathscr R\) is a
unital algebra, \(P\circ\Theta\in\mathscr R\subset\mathcal D\).

The algebra \(\mathcal P\) separates points of \(K\) and contains the
constants. By the Stone--Weierstrass theorem, \(\mathcal P\) is uniformly
dense in \(C(K)\). Hence, for every \(\varphi\in C(K)\), there exist
\(P_n\in\mathcal P\) such that \(P_n\to\varphi\) uniformly on \(K\). Then
\(P_n\circ\Theta\to\varphi\circ\Theta\) pointwise on \(\Omega\), and the
sequence is uniformly bounded. Since \(P_n\circ\Theta\in\mathcal D\) and
\(\mathcal D\) is closed under bounded pointwise a.e. limits, we obtain
\(\varphi\circ\Theta\in\mathcal D\) for every \(\varphi\in C(K)\).

Define \(\mathcal C:=\{B\in\mathcal B(K):
\mathbf 1_{\Theta^{-1}(B)}\in\mathcal D\}\).
We first show that \(\mathcal C\) contains all closed subsets of \(K\). The
empty set is trivial. If \(F\subset K\) is nonempty and closed, define
\(\varphi_n(x):=\max\{1-n\,d(x,F),0\}\), where \(d\) is a metric on \(K\).
Then \(\varphi_n\in C(K)\), \(0\le\varphi_n\le1\), and
\(\varphi_n\to\mathbf 1_F\) pointwise on \(K\). Therefore
\(\varphi_n\circ\Theta\to\mathbf 1_{\Theta^{-1}(F)}\) pointwise on
\(\Omega\).
Since \(\varphi_n\circ\Theta\in\mathcal D\) and the sequence is bounded, we get
\(\mathbf 1_{\Theta^{-1}(F)}\in\mathcal D\). Thus \(F\in\mathcal C\).

Next, \(\mathcal C\) is a Dynkin class. It contains \(K\), because
\(\mathbf 1_{\Theta^{-1}(K)}=1\in\mathscr R\subset\mathcal D\).
If \(A,B\in\mathcal C\) and \(B\subset A\), then
\(\mathbf 1_{\Theta^{-1}(A\setminus B)}
=\mathbf 1_{\Theta^{-1}(A)}-\mathbf 1_{\Theta^{-1}(B)}\in\mathcal D\).
Hence \(A\setminus B\in\mathcal C\). If \(B_1,B_2,\ldots\in\mathcal C\) are
pairwise disjoint, then
\(\sum_{m=1}^n\mathbf 1_{\Theta^{-1}(B_m)}\in\mathcal D\). This sequence is
bounded by \(1\) and converges pointwise to
\(\mathbf 1_{\Theta^{-1}(\bigcup_{m=1}^\infty B_m)}\). Thus
\(\bigcup_{m=1}^\infty B_m\in\mathcal C\), so \(\mathcal C\) is a Dynkin
class.

Since the closed subsets of \(K\) form a \(\pi\)-system generating
\(\mathcal B(K)\), the \(\pi\)-\(\lambda\) theorem gives
\(\mathcal C=\mathcal B(K)\). Hence
\(\mathbf 1_{\Theta^{-1}(B)}\in\mathcal D\) for every
\(B\in\mathcal B(K)\). Since
\(\mathcal A_0=\Theta^{-1}(\mathcal B(K))\), we have
\(\mathbf 1_A\in\mathcal D\) for every \(A\in\mathcal A_0\).

Now let \(A\in\mathcal A\). Since \(\mathcal A\) is generated by
\(\mathcal A_0\) and all \(\mu\)-null sets, there exists
\(A_0\in\mathcal A_0\) such that
\(\mu(A\triangle A_0)=0\). Thus \(\mathbf 1_A=\mathbf 1_{A_0}\) in
\(L^\infty(\Omega)\). Since \(\mathbf 1_{A_0}\in\mathcal D\), we also have
\(\mathbf 1_A\in\mathcal D\).
Therefore every bounded \(\mathcal A\)-measurable simple function belongs to
\(\mathcal D\).

Finally, let \(f\in L^\infty(\mathcal A)\). There exist bounded
\(\mathcal A\)-measurable simple functions \(s_n\) such that \(s_n\to f\)
pointwise a.e. and \(\sup_n\|s_n\|_\infty<\infty\).
Since each \(s_n\in\mathcal D\), the closure assumption gives
\(f\in\mathcal D\). Thus \(L^\infty(\mathcal A)\subset\mathcal D\). The
reverse inclusion is assumed. Hence \(\mathcal D=L^\infty(\mathcal A)\).
\end{proof}

We shall also use the following finite-product version of
Lemma~\ref{lem:bounded-pointwise-closure}.

\begin{lemma}
Let \((\Omega,\mathcal F,\mu)\) be a complete probability space. Let
\(\mathscr R\subset L^\infty(\Omega)\) be a countably generated unital
subalgebra, and put \(\mathcal A:=\sigma(\mathscr R)\). Let
\(m\in\mathbb N\), and let
\(\mathcal D\subset\bigl(L^\infty(\mathcal A)\bigr)^m\) satisfy the following
three conditions.

\begin{enumerate}
\item \(\mathscr R^m\subset\mathcal D\).

\item For each \(j=1,\ldots,m\), and for every fixed choice of
\(g_i\in L^\infty(\mathcal A)\) (\(i\neq j\)), the section
\(\{g\in L^\infty(\mathcal A):
(g_1,\ldots,g_{j-1},g,g_{j+1},\ldots,g_m)\in\mathcal D\}\)
is a vector subspace of \(L^\infty(\mathcal A)\).

\item For each \(j=1,\ldots,m\), the sections in \((2)\) are closed under
bounded pointwise a.e. limits.
\end{enumerate}

Then \(\mathcal D=\bigl(L^\infty(\mathcal A)\bigr)^m\).
\end{lemma}

\begin{proof}
Fix \(r_2,\ldots,r_m\in\mathscr R\) and consider the first-coordinate
section \(\mathcal S:=\{g\in L^\infty(\mathcal A):
(g,r_2,\ldots,r_m)\in\mathcal D\}\). By the assumptions on
\(\mathcal D\), the section \(\mathcal S\) satisfies the hypotheses of the
preceding lemma. Hence
\(\mathcal S=L^\infty(\mathcal A)\). Thus
\(L^\infty(\mathcal A)\times\mathscr R^{m-1}\subset\mathcal D\).
Next fix \(g_1\in L^\infty(\mathcal A)\) and
\(r_3,\ldots,r_m\in\mathscr R\). Applying the same argument to the
second-coordinate section, using the inclusion just proved, gives
\(\bigl(L^\infty(\mathcal A)\bigr)^2\times\mathscr R^{m-2}
\subset\mathcal D\).
Repeating this argument for the remaining coordinates, we obtain
\(\bigl(L^\infty(\mathcal A)\bigr)^m\subset\mathcal D\).
\end{proof}

\begin{lemma}
\label{lem:rooted-integral-measurability}
Let \(Y\) be a separable Banach space, let \(Y_0\subset Y\) be countable and
norm dense, and let \(U:[0,1]^2\to Y^*\) be a bounded weak-* graphon. Put
\(\mathcal A_U:=\sigma(\mathscr R_U(Y_0))\).

Let \((\mathbf H;\rho_1,\rho_2)\) be a nonadjacent two-labelled
\(Y\)-decorated test graph, and write \(\mathbf H=(H,h)\). Enumerate the
edges as \(E(H)=\{e_1,\ldots,e_N\}\). Write \(e_j=a_jb_j\) and
\(h_j:=h_{e_j}\).

Let \(0\le r\le N\), and let
\(g_1,\ldots,g_r\in L^\infty(\mathcal A_U)\). Define, for a.e.
\((x_{\rho_1},x_{\rho_2})\in[0,1]^2\),
\[
\begin{aligned}
C_r(x_{\rho_1},x_{\rho_2})
:=
\int_{[0,1]^{V(H)\setminus\{\rho_1,\rho_2\}}}
&
\left(
\prod_{i=1}^{r}
g_i(x_{a_i},x_{b_i})
\right) \\
&\times
\left(
\prod_{j=r+1}^{N}
\langle h_j,U(x_{a_j},x_{b_j})\rangle
\right)
\prod_{i\in V(H)\setminus\{\rho_1,\rho_2\}} dx_i.
\end{aligned}
\]
Then \(C_r\in L^\infty(\mathcal A_U)\).\footnote{Again, empty products in the
integrand are interpreted as \(1\).}
\end{lemma}

\begin{proof}
We first prove the assertion under the additional assumption that \(\mathbf H\)
is a \(Y_0\)-decorated test graph. For \(r=0\), we have
\(C_0=\tau_{\mathbf H}^U\), and hence
\(C_0\in\mathscr R_U(Y_0)\subset L^\infty(\mathcal A_U)\).

Now fix \(1\le r\le N\). Let \(\mathcal D\) be the set of all tuples
\((g_1,\ldots,g_r)\in\bigl(L^\infty(\mathcal A_U)\bigr)^r\) such that the
corresponding function \(C_r\) belongs to \(L^\infty(\mathcal A_U)\). We show
that \(\mathcal D\) satisfies the hypotheses of the coordinatewise closure
lemma with \(\mathscr R=\mathscr R_U(Y_0)\), \(\mathcal A=\mathcal A_U\),
and \(m=r\).

First take \(g_i=\tau_{\mathbf K_i}^U\) for \(i=1,\ldots,r\), where each
\(\mathbf K_i\) is a nonadjacent two-labelled \(Y_0\)-decorated test graph.
By the finite edge replacement lemma, \(C_r\in\mathscr R_U(Y_0)\). Hence
\((\tau_{\mathbf K_1}^U,\ldots,\tau_{\mathbf K_r}^U)\in\mathcal D\). By
linearity of the integral, after expanding finite linear combinations, we get
\(\mathscr R_U(Y_0)^r\subset\mathcal D\). The coordinate sections of
\(\mathcal D\) are vector spaces, again by linearity of the integral.

They are also closed under bounded pointwise a.e. limits. Fix all coordinates
except the \(i\)-th one, and suppose that the tuples obtained by replacing
\(g_i\) with \(g_i^{(n)}\) belong to \(\mathcal D\), where
\(g_i^{(n)}\to g_i\) pointwise a.e. on \([0,1]^2\) and
\(\sup_n\|g_i^{(n)}\|_\infty<\infty\). Let \(C_r^{(n)}\) be the function
obtained by using \(g_i^{(n)}\) instead of \(g_i\). For a.e. fixed external
variables \((x_{\rho_1},x_{\rho_2})\), the corresponding integrands converge
pointwise a.e. to the integrand defining \(C_r\). These integrands are
uniformly \(L^\infty\)-bounded, since \(U\) and \(g_j\) for \(j\neq i\) are
bounded and \(g_i^{(n)}\) is uniformly \(L^\infty\)-bounded. Hence the
dominated convergence theorem gives \(C_r^{(n)}\to C_r\) pointwise a.e. on
\([0,1]^2\).

Each \(C_r^{(n)}\) is \(\mathcal A_U\)-measurable. Since \(\mathcal A_U\)
is complete, the a.e. limit \(C_r\) is also \(\mathcal A_U\)-measurable.
Moreover, \(C_r\) is bounded by the boundedness of \(U\) and of
\(g_1,\ldots,g_r\). Hence \(C_r\in L^\infty(\mathcal A_U)\). Thus the
sections are closed under bounded pointwise a.e. limits. By the coordinatewise
closure lemma,
\(\mathcal D=\bigl(L^\infty(\mathcal A_U)\bigr)^r\). Therefore the assertion
holds for every \(r=0,1,\ldots,N\) when \(\mathbf H\) is \(Y_0\)-decorated.

Now let \(\mathbf H\) be \(Y\)-decorated. For each edge \(e_j\), choose
\(h_j^{(n)}\in Y_0\) such that \(h_j^{(n)}\to h_j\) in norm. Since each
sequence is convergent, it is bounded. Let \(\mathbf H^{(n)}\) be the
\(Y_0\)-decorated test graph with the same underlying graph, the same labels,
and edge decorations \(h_j^{(n)}\) on \(e_j\). For the fixed \(r\), let
\(C_r^{(n)}\) be the function obtained from \(\mathbf H^{(n)}\) by the same
formula as \(C_r\). By the \(Y_0\)-decorated case already proved,
\(C_r^{(n)}\in L^\infty(\mathcal A_U)\) for every \(n\ge1\).

Since \(h_j^{(n)}\to h_j\) in norm and \(U\) is bounded, the corresponding
edge factors converge pointwise a.e. For a.e. fixed external variables, the
integrands defining \(C_r^{(n)}\) converge pointwise a.e. to the integrand
defining \(C_r\). These integrands are uniformly \(L^\infty\)-bounded,
because \(U\), the \(g_i\)'s, and the approximating decorations
\(h_j^{(n)}\) are bounded for all \(j\). Hence the dominated convergence
theorem gives \(C_r^{(n)}\to C_r\) pointwise a.e. on \([0,1]^2\).

Each \(C_r^{(n)}\) is \(\mathcal A_U\)-measurable. Since \(\mathcal A_U\)
is complete, \(C_r\) is also \(\mathcal A_U\)-measurable. Finally, \(C_r\)
is bounded by the boundedness of \(U\), of \(g_1,\ldots,g_r\), and of the
finitely many decorations \(h_1,\ldots,h_N\). Hence
\(C_r\in L^\infty(\mathcal A_U)\). This proves the lemma.
\end{proof}

\subsection{Proof of the bounded characterization}

In this subsection, we denote Lebesgue measure on \([0,1]^2\) by \(\mu\).
Recall that, for a Banach space \(X\), a subspace \(Y\subset X^*\) is \emph{norming} if
\(\|x\|=\sup_{y\in B_Y}|\langle y,x\rangle|\) for every \(x\in X\).
For such a subspace, the canonical map \(J:X\to Y^*\), defined by
\(\langle Jx,y\rangle=\langle y,x\rangle\), is an isometry, and
\(Y\) is weak-* dense in \(X^*\). If \(X\) is separable, then
\(B_{X^*}\) is weak-* compact and metrizable. Hence we may choose a
countable weak-* dense subset \(D\subset B_{X^*}\) and let
\(Y=\overline{\operatorname{span}}D\). Since \(D\subset B_Y\), this gives a
closed separable norming subspace \(Y\subset X^*\).

\begin{lemma}
Let \(X\) be a separable weakly sequentially complete Banach space, let
\(Y\subset X^*\) be a closed separable norming subspace, and let
\(J:X\to Y^*\) be the canonical map. Let
\(Y_0\subset Y\) be a countable norm-dense subset of \(Y\), and let
\((\mathbf G_n)\) be a sequence of \(X\)-decorated target graphs
satisfying the two assumptions in Definition~\ref{def:bgrp}.

Then there exists a bounded \(Y^*\)-valued weak-* graphon
\(U:[0,1]^2\to Y^*\) such that
\begin{enumerate}
\item For every \(Y\)-decorated test graph \(\mathbf F\), we have
\(t(\mathbf F,\mathbf G_n)\to t(\mathbf F,U)\).
\item For every \(r\in L^\infty(\mathcal A_U)\), we have
\(\int_{[0,1]^2}rU\,d\mu\in J(X)\), where
\(\mathcal A_U:=\sigma(\mathscr R_U(Y_0))\).
\end{enumerate}
Here the integral in (2) is understood in the weak-* sense.
\end{lemma}

\begin{proof}
Set \(M:=\sup_n\|\mathbf G_n\|_\infty\). For each \(n\), let
\(W_n:[0,1]^2\to X\) be the equal-step graphon associated with
\(\mathbf G_n\). Then \(\|W_n\|_{L^\infty(X)}\le M\). Consequently,
\(JW_n:[0,1]^2\to Y^*\) is a \(Y^*\)-valued weak-* graphon satisfying
\(\|JW_n\|_{L^\infty(Y^*)}\le M\).

Every \(Y\)-decorated test graph is also an \(X^*\)-decorated test graph.
Hence, for every \(Y\)-decorated test graph \(\mathbf F\), the sequence
\(t(\mathbf F,\mathbf G_n)\) converges. Since
\(t(\mathbf F,JW_n)=t(\mathbf F,W_n)=t(\mathbf F,\mathbf G_n)\), the sequence
\(t(\mathbf F,JW_n)\) also converges. It follows from
Theorem~\ref{thm:kls-3.7} that there exists a \(Y^*\)-valued weak-* graphon
\(U:[0,1]^2\to Y^*\) such that \(t(\mathbf F,JW_n)\to t(\mathbf F,U)\) for
every \(Y\)-decorated test graph \(\mathbf F\).

We next show that \(\|U\|_{L^\infty(Y^*)}\le M\). The construction in
Theorem~\ref{thm:kls-3.7} gives \(\|U\|_{L^p(Y^*)}\le M\) for every
\(1\le p<\infty\). For \(\varepsilon>0\), put
\[
E_\varepsilon
:=
\{(s,t)\in[0,1]^2:\|U(s,t)\|_{Y^*}>M+\varepsilon\}.
\]
Then, for every \(1\le p<\infty\),
\[
(M+\varepsilon)^p\mu(E_\varepsilon)
\le
\int_{E_\varepsilon}\|U(s,t)\|_{Y^*}^p\,d\mu(s,t)
\le
\|U\|_{L^p(Y^*)}^p
\le
M^p.
\]
Therefore
\(\mu(E_\varepsilon)\le(M/(M+\varepsilon))^p\). Letting \(p\to\infty\), we
obtain \(\mu(E_\varepsilon)=0\). Since this holds for every
\(\varepsilon>0\), it follows that \(\|U\|_{L^\infty(Y^*)}\le M\).

We now prove that \(\int rU\,d\mu\in J(X)\) for every
\(r\in\mathscr R_U(Y_0)\). By the definition of \(\mathscr R_U(Y_0)\), it
suffices to consider \(r=\tau_{\mathbf F}^U\), where
\((\mathbf F;\rho_1,\rho_2)\) is a nonadjacent two-labelled
\(Y_0\)-decorated test graph. Relabelling the vertices if necessary, we may
assume that the labelled vertices are \(1\) and \(2\).

Since \(U\) and \(\tau_{\mathbf F}^U\) are bounded, the weak-* integral
\(\int_{[0,1]^2}\tau_{\mathbf F}^U(s,t)U(s,t)\,d\mu(s,t)\in Y^*\) is well
defined by
\(\langle y,\int_{[0,1]^2}\tau_{\mathbf F}^U U\,d\mu\rangle
=\int_{[0,1]^2}\tau_{\mathbf F}^U(s,t)\langle y,U(s,t)\rangle\,d\mu(s,t)\)
for \(y\in Y\).
For each \(n\), define
\(a_n:=\int_{[0,1]^2}\tau_{\mathbf F}^{W_n}(s,t)W_n(s,t)\,d\mu(s,t)\in X\).
Since \(W_n\) is an \(X\)-valued simple function and
\(\tau_{\mathbf F}^{W_n}\) is a bounded measurable scalar function, their
product is Bochner measurable. It is also bounded, and hence Bochner
integrable. Consequently, \(a_n\) is well defined.

For every \(x^*\in X^*\), the definition of the Bochner integral gives
\begin{equation}
\label{eq:root-slice-density}
\langle x^*,a_n\rangle
=
\int_{[0,1]^2}
\tau_{\mathbf F}^{W_n}(s,t)
\langle x^*,W_n(s,t)\rangle\,d\mu(s,t)
=
t(\mathbf F+e_{12}^{x^*},W_n).
\end{equation}
Here \(\mathbf F+e_{12}^{x^*}\) is obtained by adding the edge between the two
labelled vertices and decorating it by \(x^*\). Since \(\mathbf F\) is
nonadjacent, this is a well-defined \(X^*\)-decorated test graph. Moreover,
\(t(\mathbf F+e_{12}^{x^*},W_n)
=t(\mathbf F+e_{12}^{x^*},\mathbf G_n)\). The right-hand side converges by
assumption for every \(x^*\in X^*\). Hence, it follows from
\eqref{eq:root-slice-density} that \((a_n)\) is weakly Cauchy in \(X\).

Since \(X\) is weakly sequentially complete, there exists \(a\in X\) such that
\(a_n\rightharpoonup a\) weakly in \(X\). Now take \(y\in Y\). Then
\(\mathbf F+e_{12}^{y}\) is a \(Y\)-decorated test graph. Using
\eqref{eq:root-slice-density} with \(x^*=y\), we obtain
\(\langle y,a\rangle=\lim_{n\to\infty}\langle y,a_n\rangle
=\lim_{n\to\infty}t(\mathbf F+e_{12}^{y},W_n)
=\lim_{n\to\infty}t(\mathbf F+e_{12}^{y},JW_n)
=t(\mathbf F+e_{12}^{y},U)\). Expanding the last density gives
\(t(\mathbf F+e_{12}^{y},U)
=\int_{[0,1]^2}\tau_{\mathbf F}^U(s,t)
\langle y,U(s,t)\rangle\,d\mu(s,t)\). By the definition of the weak-*
integral, this means that
\(\langle y,a\rangle
=\langle y,\int_{[0,1]^2}\tau_{\mathbf F}^U U\,d\mu\rangle\) for
\(y\in Y\). Equivalently,
\(J(a)=\int_{[0,1]^2}\tau_{\mathbf F}^U U\,d\mu\) in \(Y^*\). Thus
\(\int_{[0,1]^2}\tau_{\mathbf F}^U U\,d\mu\in J(X)\). By linearity, the same
conclusion holds for every \(r\in\mathscr R_U(Y_0)\), namely
\(\int_{[0,1]^2}rU\,d\mu\in J(X)\).

We now extend the conclusion from \(\mathscr R_U(Y_0)\) to
\(L^\infty(\mathcal A_U)\). Let \(\mathcal D\) be the set of all
\(r\in L^\infty(\mathcal A_U)\) such that
\(\int_{[0,1]^2}rU\,d\mu\in J(X)\). The set \(\mathcal D\) is a vector
subspace of \(L^\infty(\mathcal A_U)\), by linearity of the weak-* integral
and the fact that \(J(X)\) is a vector subspace of \(Y^*\). The preceding
argument shows that \(\mathscr R_U(Y_0)\subset\mathcal D\).

It remains to verify that \(\mathcal D\) is closed under bounded pointwise
a.e. limits. Let \(r_n\in\mathcal D\) satisfy
\(\sup_n\|r_n\|_\infty<\infty\), and suppose that \(r_n\to r\) pointwise a.e.
for some \(r\in L^\infty(\mathcal A_U)\). By the dominated convergence
theorem, \(\|r_n-r\|_{L^1}\to0\). Moreover,
\(\|\int_{[0,1]^2}(r_n-r)U\,d\mu\|_{Y^*}
\le\|U\|_{L^\infty(Y^*)}\|r_n-r\|_{L^1}\), and hence
\(\int_{[0,1]^2}r_nU\,d\mu\to\int_{[0,1]^2}rU\,d\mu\) in the norm of \(Y^*\).

Since \(J(X)\) is norm closed in \(Y^*\) and each
\(\int_{[0,1]^2}r_nU\,d\mu\) belongs to \(J(X)\), the norm limit
\(\int_{[0,1]^2}rU\,d\mu\) also belongs to \(J(X)\). Thus
\(r\in\mathcal D\).

Since \(\mathscr R_U(Y_0)\) is a countably generated unital subalgebra and
\(\mathcal A_U=\sigma(\mathscr R_U(Y_0))\), the subspace
\(\mathcal D\subset L^\infty(\mathcal A_U)\) satisfies the hypotheses of
Lemma~\ref{lem:bounded-pointwise-closure}. Therefore
\(\mathcal D=L^\infty(\mathcal A_U)\). Consequently,
\(\int_{[0,1]^2}rU\,d\mu\in J(X)\) for every
\(r\in L^\infty(\mathcal A_U)\). This proves condition (2) and completes the
proof.
\end{proof}

\begin{lemma}
\label{lem:rnp-y-test-representation}
Let \(X\) be a separable Banach space which is weakly sequentially complete
and has the Radon--Nikodým property. Let \(Y\subset X^*\) be a closed separable
norming subspace, and let \((\mathbf G_n)\) be a sequence of
\(X\)-decorated target graphs satisfying the two assumptions in
Definition~\ref{def:bgrp}.

Then there exists a bounded \(X\)-valued Bochner graphon
\(W:[0,1]^2\to X\) such that
\(t(\mathbf F,\mathbf G_n)\to t(\mathbf F,W)\) for every
\(Y\)-decorated test graph \(\mathbf F\).
\end{lemma}

\begin{proof}
Choose a countable norm-dense subset \(Y_0\subset Y\), and let
\(J:X\to Y^*\) be the canonical embedding defined by
\(\langle Jx,y\rangle=\langle y,x\rangle\).
By the preceding lemma, there exists a bounded \(Y^*\)-valued weak-* graphon
\(U:[0,1]^2\to Y^*\) such that
\(t(\mathbf F,\mathbf G_n)\to t(\mathbf F,U)\) for every
\(Y\)-decorated test graph \(\mathbf F\), and
\(\int_{[0,1]^2}rU\,d\mu\in J(X)\) for every
\(r\in L^\infty(\mathcal A_U)\), where
\(\mathcal A_U:=\sigma(\mathscr R_U(Y_0))\).
Set \(M:=\|U\|_{L^\infty(Y^*)}\). For \(A\in\mathcal A_U\), define
\[
\nu(A):=J^{-1}\left(\int_AU\,d\mu\right).
\]
This is well defined because \(\mathbf 1_A\in L^\infty(\mathcal A_U)\).
Since \(J:X\to Y^*\) is an isometry,
\[
\begin{aligned}
\|\nu(A)\|_X
&=
\left\|\int_AU\,d\mu\right\|_{Y^*} \\
&\le
\int_A\|U(s,t)\|_{Y^*}\,d\mu(s,t)
\le M\mu(A).
\end{aligned}
\]

This estimate also shows that \(\nu\) is countably additive in the norm of
\(X\). Indeed, if \(A_1,A_2,\ldots\in\mathcal A_U\) are pairwise disjoint
and \(A=\bigcup_{k=1}^\infty A_k\), then finite additivity and the preceding
estimate give
\[
\left\|
\nu(A)-\sum_{k=1}^n\nu(A_k)
\right\|_X
\le
M\mu\left(\bigcup_{k>n}A_k\right)
\longrightarrow0.
\]
Moreover, for every finite measurable partition \(A=\bigcup_iA_i\),
\[
\sum_i\|\nu(A_i)\|_X
\le
M\sum_i\mu(A_i)
=
M\mu(A).
\]
Hence \(\nu\) has bounded variation, \(|\nu|(A)\le M\mu(A)\), and
\(\nu\ll\mu\).

Since \(X\) has the Radon--Nikodým property, there exists an
\(\mathcal A_U\)-Bochner measurable function \(W:[0,1]^2\to X\) such that
\(\nu(A)=\int_AW\,d\mu\) for every \(A\in\mathcal A_U\). Furthermore, the
variation formula for a vector measure with Bochner density
\cite[(2.2), p.~43]{pisier2016martingales} gives
\[
\int_A\|W\|_X\,d\mu
=
|\nu|(A)
\le M\mu(A).
\]
It follows that \(\|W\|_X\le M\) a.e. Hence \(W\) is bounded.

We next verify that \(\vartheta^{-1}(A)\in\mathcal A_U\) for every
\(A\in\mathcal A_U\), where \(\vartheta(s,t):=(t,s)\). Let
\[
\mathcal C
:=
\{A\in\mathcal A_U:\vartheta^{-1}(A)\in\mathcal A_U\}.
\]
Since inverse images commute with complements and countable unions,
\(\mathcal C\) is a sub-sigma-algebra of \(\mathcal A_U\).
Let \((\mathbf F;\rho_1,\rho_2)\) be a nonadjacent two-labelled
\(Y_0\)-decorated test graph, and put
\(\mathbf F':=(\mathbf F;\rho_2,\rho_1)\). Then
\(\tau_{\mathbf F}^U\circ\vartheta=\tau_{\mathbf F'}^U\) a.e.

Consequently, for every Borel set \(B\subset\mathbb R\), the set
\(\vartheta^{-1}((\tau_{\mathbf F}^U)^{-1}(B))\) differs from
\((\tau_{\mathbf F'}^U)^{-1}(B)\) by a null set. Since \(\mathcal A_U\)
contains \((\tau_{\mathbf F'}^U)^{-1}(B)\) and all null sets, it follows
that \((\tau_{\mathbf F}^U)^{-1}(B)\in\mathcal C\). Moreover, if
\(N\subset[0,1]^2\) is a null set, then \(\vartheta^{-1}(N)\) is also null
because \(\vartheta\) preserves Lebesgue measure. Thus \(\mathcal C\)
contains all null sets. Since the sigma-algebra \(\mathcal A_U\) is
generated by the sets \((\tau_{\mathbf F}^U)^{-1}(B)\) and the null sets,
it follows that \(\mathcal C=\mathcal A_U\).

Consequently, \(W\circ\vartheta\) is also \(\mathcal A_U\)-Bochner
measurable. For every \(A\in\mathcal A_U\), the invariance of Lebesgue
measure and the symmetry of \(U\) give
\[
\begin{aligned}
\int_AW\circ\vartheta\,d\mu
&=
\int_{\vartheta^{-1}(A)}W\,d\mu \\
&=
J^{-1}\left(\int_{\vartheta^{-1}(A)}U\,d\mu\right) \\
&=
J^{-1}\left(\int_AU\circ\vartheta\,d\mu\right) \\
&=
J^{-1}\left(\int_AU\,d\mu\right)
=
\int_AW\,d\mu.
\end{aligned}
\]
Thus \(W\) and \(W\circ\vartheta\) are Bochner densities of the same vector
measure. By uniqueness of the Radon--Nikodým derivative,
\(W=W\circ\vartheta\) a.e. Hence \(W\) is symmetric and therefore is an
\(X\)-valued Bochner graphon.

The defining identity for \(W\) gives
\(J(\int_AW\,d\mu)=\int_AU\,d\mu\) for every \(A\in\mathcal A_U\). By
linearity, the same identity holds with \(\mathbf 1_A\) replaced by an
\(\mathcal A_U\)-measurable simple function. Approximating an arbitrary
bounded \(\mathcal A_U\)-measurable function by bounded simple functions,
we obtain
\[
J\left(\int_{[0,1]^2}CW\,d\mu\right)
=
\int_{[0,1]^2}CU\,d\mu
\qquad
\bigl(C\in L^\infty(\mathcal A_U)\bigr).
\]
Evaluating this identity at \(y\in Y\), we get
\[
\int_{[0,1]^2}C(s,t)\langle y,W(s,t)\rangle\,d\mu(s,t)
=
\int_{[0,1]^2}C(s,t)\langle y,U(s,t)\rangle\,d\mu(s,t).
\tag{*}
\]

It remains to compare the density of a \(Y\)-decorated test graph
\(\mathbf F\) in \(U\) and \(W\). Write \(\mathbf F=(F,f)\), enumerate its
edges as \(E(F)=\{e_1,\ldots,e_N\}\), and write \(e_j=a_jb_j\) and
\(y_j:=f_{e_j}\). For \(0\le r\le N\), define
\[
T_r
:=
\int_{[0,1]^{V(F)}}
\left(
\prod_{i=1}^{r}
\langle y_i,W(x_{a_i},x_{b_i})\rangle
\right)
\left(
\prod_{j=r+1}^{N}
\langle y_j,U(x_{a_j},x_{b_j})\rangle
\right)
\prod_{i\in V(F)} dx_i.
\]
Then \(T_0=t(\mathbf F,U)\) and \(T_N=t(\mathbf F,W)\). To prove
\(T_0=T_N\), it is enough to show that \(T_r=T_{r+1}\) for
\(r=0,\ldots,N-1\).

Fix \(0\le r<N\). By Fubini's theorem, the definitions of \(T_r\) and
\(T_{r+1}\) give
\[
\begin{aligned}
T_r
&=
\int_{[0,1]^2}
C_r(x_{a_{r+1}},x_{b_{r+1}})\,
\langle y_{r+1},U(x_{a_{r+1}},x_{b_{r+1}})\rangle
\,dx_{a_{r+1}}\,dx_{b_{r+1}}, \\
T_{r+1}
&=
\int_{[0,1]^2}
C_r(x_{a_{r+1}},x_{b_{r+1}})\,
\langle y_{r+1},W(x_{a_{r+1}},x_{b_{r+1}})\rangle
\,dx_{a_{r+1}}\,dx_{b_{r+1}},
\end{aligned}
\]
where
\[
\begin{aligned}
C_r(x_{a_{r+1}},x_{b_{r+1}})
&:=
\int_{[0,1]^{V(F)\setminus\{a_{r+1},b_{r+1}\}}}
\left(
\prod_{i=1}^{r}
\langle y_i,W(x_{a_i},x_{b_i})\rangle
\right) \\
&\quad\times
\left(
\prod_{j=r+2}^{N}
\langle y_j,U(x_{a_j},x_{b_j})\rangle
\right)
\prod_{v\in V(F)\setminus\{a_{r+1},b_{r+1}\}} dx_v.
\end{aligned}
\]

We next show that \(C_r\in L^\infty(\mathcal A_U)\). Let
\((\mathbf H_{r+1};a_{r+1},b_{r+1})\) be the two-labelled \(Y\)-decorated test
graph obtained from \(\mathbf F\) by deleting \(e_{r+1}\) and regarding
\(a_{r+1}\) and \(b_{r+1}\) as the two labelled vertices.
Since \(F\) is simple, this graph is nonadjacent. Enumerate the edges of
\(H_{r+1}\) in the order \(e_1,\ldots,e_r,e_{r+2},\ldots,e_N\). For
\(i=1,\ldots,r\), put
\(g_i(s,t):=\langle y_i,W(s,t)\rangle\). Since \(W\) is
\(\mathcal A_U\)-Bochner measurable, \(g_i\in L^\infty(\mathcal A_U)\).
Applying Lemma~\ref{lem:rooted-integral-measurability} to
\(\mathbf H_{r+1}\) with the functions \(g_1,\ldots,g_r\) gives
\(C_r\in L^\infty(\mathcal A_U)\).

We may therefore apply \((*)\) with \(C=C_r\) and
\(y=y_{r+1}\). It follows that \(T_r=T_{r+1}\). Repeating this for
\(r=0,\ldots,N-1\), we obtain
\[
t(\mathbf F,U)=T_0=T_1=\cdots=T_N=t(\mathbf F,W).
\]
Finally, the preceding lemma gives
\(t(\mathbf F,\mathbf G_n)\to t(\mathbf F,U)\). Therefore
\(t(\mathbf F,\mathbf G_n)\to t(\mathbf F,W)\) for every
\(Y\)-decorated test graph \(\mathbf F\). This proves the lemma.
\end{proof}

We now complete the proof of Theorem~\ref{thm:bounded-characterization}.

\begin{proof}[Proof of the bounded characterization]
Necessity follows from Proposition~\ref{prop:bounded-inherited}(1).
Conversely, assume that
\(X\) is weakly sequentially complete and has the Radon--Nikodým property. By
Proposition~\ref{prop:bounded-inherited}(2), it is enough to prove the bounded
representation property for every closed separable subspace of \(X\). Such
subspaces inherit weak sequential completeness, and the Radon--Nikodým
property passes to closed subspaces
\cite[Theorem~III.3.2]{diestelUhl1977vector}. Hence we may assume that \(X\) is
separable.

Let \((\mathbf G_n)\) be a sequence of \(X\)-decorated target graphs
satisfying the two assumptions in Definition~\ref{def:bgrp}. Since \(X\) is
separable, choose a closed separable norming subspace \(Y\subset X^*\). By
Lemma~\ref{lem:rnp-y-test-representation}, there exists a bounded
\(X\)-valued Bochner graphon \(W\) such that
\(t(\mathbf F,\mathbf G_n)\to t(\mathbf F,W)\) for every \(Y\)-decorated
test graph \(\mathbf F\). Since \(Y\) is norming, it is weak-* dense in
\(X^*\). By Lemma~\ref{lem:full-dual-upgrade}, the same
convergence holds for every \(X^*\)-decorated test graph.
Thus \(W\) represents all the density limits of \((\mathbf G_n)\). This
proves the theorem.
\end{proof}

\bibliographystyle{amsplain}
\bibliography{refs}

\end{document}